\documentclass[twoside,12pt,reqno]{amsart}
\usepackage{latexsym}
\usepackage{amssymb, cases}
\usepackage{epsfig}
\usepackage{multirow}
\usepackage{caption}
\usepackage{subcaption}
\usepackage{graphicx}
\usepackage{color,verbatim}

\newtheorem{thm}{Theorem}[section]

\newtheorem{lem}[thm]{Lemma}

\theoremstyle{definition}

\theoremstyle{remark}
\newtheorem{rem}{Remark}[section]

\makeatletter
\def\serieslogo@{}
\makeatother \makeatletter
\def\@setcopyright{}
\makeatother

\begin{document}
\title[Positivity-preserving and energy-satisfying DG methods]{Positivity-preserving and energy-dissipating discontinuous Galerkin methods for 
nonlinear nonlocal Fokker-Planck equations}

\author[]{Jos\'{e} A. Carrillo$^1$, Hailiang Liu$^2$, and Hui Yu$^3$}
\address{$^1$ Mathematical Institute, University of Oxford, Oxford, OX2 6GG, United Kingdom.}
\email{carrillo@maths.ox.ac.uk}
\address{$^2$Iowa State University, Mathematics Department, Ames, IA 50011, United States.} \email{hliu@iastate.edu}
\address{$^3$ School of Mathematics and Computational Science, Xiangtan University, Xiangtan, 411105, China.}
\email{huiyu@xtu.edu.cn}

\begin{abstract}
This paper is concerned with structure-preserving numerical approximations for a class of nonlinear nonlocal Fokker-Planck equations, which admit a gradient flow structure and find application in diverse contexts.   The solutions,  representing density distributions, must  be non-negative and satisfy a specific energy dissipation law.   
We design an arbitrary high-order discontinuous Galerkin (DG) method  tailored for these model problems. Both semi-discrete and fully discrete schemes are shown to admit the energy dissipation law for non-negative numerical solutions. To ensure the preservation of positivity in cell averages at all time steps, we introduce a local flux correction applied to the DDG diffusive flux. 
Subsequently, a hybrid algorithm is presented, utilizing a positivity-preserving limiter, to generate positive and energy-dissipating solutions. Numerical examples are provided to showcase the high resolution of the numerical solutions and the verified properties of the DG schemes.
\end{abstract}

\maketitle

\section{Introduction}
In this paper,  we develop an arbitrary high-order discontinuous Galerkin (DG) method with positivity preservation and energy dissipation for the following problem
\begin{subequations}\label{IVP}
\begin{numcases}{}
\partial_t \rho  =\nabla_x\cdot [\rho \nabla_x (V(x) + H'(\rho) + W*\rho) ],\quad x\in \Omega\subset \mathbb{R}^d,\; t >0, \\
\rho(0, x)  = \rho_0(x),
\end{numcases}
\end{subequations}
subject to appropriate boundary conditions. Here 
$\Omega$ is a bounded domain in the $d$-dimensional space,  
$ \rho(t, x)\geq 0$  is the unknown probability density,  
$H(\rho)$ is the density of internal energy, 
$W(x)$ is the so-called interaction potential,  and $V(x)$ is the confinement potential; see \cite{Vi03, CMV03}.  $H$ is a convex function so that $\nabla\cdot [\rho H''(\rho)\nabla \rho]$ is a density-dependent diffusion. The drift term $\nabla\cdot[\rho \nabla V]$ and the interaction term  $\nabla\cdot[\rho \nabla (W*\rho)]$ are respectively induced by external forces and interaction forces, determined by a smooth potential $W(x)=W(-x)$ (symmetric).  

Equation (\ref{IVP}a) often referred to as nonlinear Fokker-Planck equations has appeared in various applications such as in cell migration and
chemotaxis \cite{Ke70}, collective motion of animals (swarming) \cite{KCB13}, self-assembly of nanoparticles \cite{HP06}, and biological channels \cite{LHMZ10}, among others. When both $W$ and $V$ vanish, it reduces to the classical heat or porous medium equations with 
$H(\rho) = \rho {\rm log} \rho -\rho$ or
$H(\rho) = \rho^m (m>1)$, respectively \cite{CT00, Ot01, Va06}. In the presence of a nonlocal Newtonian interaction kernel $W$, the equation models a chemotaxis system, for which the Keller–Segel model in its classical version \cite{Ke70} or with nonlinear diffusions \cite{CC06,BCL09} is widely known. If the nonlocal interaction kernel $W$ is induced from the Poisson equation, the equation with $H(\rho) = \rho {\rm log} \rho -\rho$  is related to the Poisson–Nernst–Planck system \cite{LHMZ10}. 

This nonlinear Fokker-Planck equation has a gradient flow structure,  as discovered in \cite{CMV03}, of the free energy 
\begin{equation}\label{E_t}
E[\rho]=\int_{\Omega} (\rho V(x) +H(\rho))\,dx +\frac{1}{2} \int_\Omega \int_\Omega W(x-y)\rho(x)\rho(y)\,dxdy.
\end{equation}
A simple calculation shows that the evolution of this free-energy functional along a solution of \eqref{IVP} with zero flux boundary conditions is given by 
\begin{equation}\label{dEdt}
    \frac{d}{dt} E[\rho](t) =-\int_{\Omega} \rho |\nabla_x(V+H'(\rho) +W*\rho)|^2\,dx \leq 0,
\end{equation}
which is referred to as the energy dissipation property of the underlying system. 
Another two companion properties include solution positivity (non-negativity, more precisely) and mass conservation, i.e.,
\begin{equation}\label{positivity}
    \rho_0(x)\geq 0  \Longrightarrow \rho(t, x)\geq 0 \quad \text{ for }  t>0,
\end{equation}
\begin{equation}\label{mass}
    \int_{\Omega}\rho(t, x)dx=\int_{\Omega}\rho_0(x)dx \quad \text{ for }  t>0.
\end{equation}
In order to capture the rich dynamics of solutions to (\ref{IVP}), it is highly desirable to develop high-order schemes which can preserve the energy dissipation law (\ref{dEdt}), solution positivity (\ref{positivity}), and mass conservation (\ref{mass})
at the discrete level. 

The design of structure-preserving schemes for this model, and  other models of similar nature, has gained increasing attention in recent years.  
In \cite{LY12}, first order implicit numerical schemes were developed for linear (yet singular) Fokker-Planck equations, ensuring the satisfaction of all three solution properties. This approach was subsequently extended in \cite{LW14}  to a system of Poisson-Nernst-Planck equations. The preservation of three solution properties is ensured if $\tau=\mathcal{O}(h^2)$, where $\tau$ represents the time step, and $h$ is the spatial mesh size. Notably, such time step restriction has been eliminated through implicit-explicit time discretization in \cite{LM20b,LM21}, while still maintaining all three solution properties. In \cite{LM20a}, the authors designed and analyzed second order (in both space and time) positivity-preserving and free energy dissipating schemes for (\ref{IVP}),  with a mild time step restriction of size $\mathcal{O}(1)$. A comparable  finite difference scheme in the multi-dimensional case was studied in \cite{QWZ2019}. Proposals for a finite volume method for \eqref{IVP} without the interaction kernel $W$ were introduced in \cite{BF2012}, further generalized in \cite{CCH14} to cover scenarios with nonlocal interaction. Positivity is enforced using piecewise linear polynomials interpolating interface values under a CFL condition $\tau=O(h^2)$. Extensions to equations and systems with implicit time discretization and saturation terms can be found in \cite{BCH2020,BCH2023},  where the three major solution properties persist at the discrete level. Various numerical methods, including the energetic variational approach presented in \cite{DCLYZ2021}, have been introduced to preserve the three major solution properties in this context.  

It has been a challenging task to design higher-order schemes (beyond  second order) while preserving all three solution properties for \eqref{IVP}. In addressing this challenge, discontinuous Galerkin (DG) methods have captured increasing attention. The DG method belongs to  the class of finite element methods, using a completely discontinuous piecewise polynomial space for both the numerical solution and test functions. 
A key advantage of the DG method lies in its flexibility, achieved through the use of local approximation spaces
and the thoughtful design of numerical fluxes across computational cell interfaces. This flexibility enables the method to effectively capture and preserve the solution properties of the given equation. For a more 
comprehensive understanding of DG methods applied to  elliptic, parabolic, and hyperbolic PDEs,  
additional information about can be found in relevant literature such as books and lecture notes, as referenced in \cite{HW07, Ri08, Sh09}.
 
The DG discretization employed in this study draws inspiration from  the direct DG (DDG) method, as  proposed in \cite{LY09, LY10}. A distinguishing feature of this approach is the selection of numerical fluxes for the solution gradient, incorporating higher-order derivatives evaluated across cell interfaces. This characteristic proves particularly advantageous in preserving global solution properties, such as energy dissipation.  Within the DDG framework, ongoing research focus on exploring schemes with provable local or pointwise solution  properties, including positivity preservation.  A powerful strategy involves the use of a suitable reformulation of the underlying PDE model. Relevant works in this domain include \cite{LY14, LY15, LW16, LW17, SCS18}. In \cite{LY15}, an approach is applied to  the linear Fokker-Planck equation:
\begin{equation}\label{fplinear+}
\partial_t \rho =\nabla_x\cdot (\nabla_x \rho  + \rho \nabla_x V).
\end{equation}
This equation corresponds to (\ref{IVP}) with $H(\rho)=\rho \log \rho$ and $W=0$. The authors developed entropy-satisfying DDG schemes of arbitrary high order for (\ref{fplinear+}),  based on the non-logarithmic Landau formulation:
$$
\partial_t \rho =\nabla_x\cdot \left[M\nabla_x \left(\frac{\rho}{M}\right)\right] \quad \text{ with } M(x)=e^{-V(x)}.
$$
This formulation ensures satisfaction of the quadratic entropy dissipation law.  The higher-order method in \cite{LY15} extends and improves upon the finite volume method introduced in \cite{LY12}. Additionally, based on this reformulation, a third-order maximum-principle-preserving DG scheme was developed in \cite{LY14}. 
These third-order DDG schemes have been further extended to solve convection-diffusion equations with anisotropic diffusivity in \cite{YL19}, where a main difficulty stems from the anisotropic diffusion. An interesting extension to the Poisson-Nernst-Planck equations was explored in \cite{LWYY22}. Notably, the results based on this  reformulation appear to be restricted to spatial accuracy
up to the third order.  

Employing a different reformulation in terms of the energy flux, the authors in \cite{LW17} designed free energy satisfying DG schemes for Poisson-Nernst-Planck equations at any high order. 
While solution positivity is maintained through the application of a positivity-preserving limiter; unfortunately the assurance of positive cell averages,  essentially required for the limiter,  has only been proved  for specific cases thus far. 

In \cite{SCS18}, through a local reformulation of \eqref{IVP} as: 
 \begin{align*}
    \left\{\begin{array}{cl}
q & =V(x)+H'(\rho) +W*\rho,\\
\xi & =\partial_x q,\\
\partial_t \rho & =\partial_x (\rho \xi), 
    \end{array}\right.
\end{align*}
a high order nodal local DG (LDG) method was constructed using $k+1$ Gauss--Lobatto quadrature points for degree $k$ polynomials. This method aimed to preserve both energy dissipation and solution positivity, aided by a positivity-preserving limiter.  However, in some test cases, a deterioration in accuracy was observed. 

In this paper, we adopt the concept of weak positivity developed in \cite{ZS10,ZXS12, Zh17}.  This strategy requires to ensure the positivity of cell averages in the context of the forward Euler time discretization. A simple efficient scaling limiter can then be used to modify negative point values to nonnegative ones without changing cell averages. 

To address an outstanding issue from \cite{LW17}, we propose a local flux correction applied to the DDG diffusive flux. This resolution enables the energy-satisfying DG (ESDG) method to attain arbitrarily high  order. Specifically, the discrete level adheres to the energy dissipation law (\ref{dEdt}),  and the numerical solutions are ensured to remain non-negative. The key concept involves projecting the energy flux, represented 
by $H'(\rho) + V(x) + W*\rho$ onto the numerical solution space at each time step. Subsequently, both the convective and diffusive components of equation (\ref{IVP}) are discretized using the obtained projected energy flux. The establishment of provable positivity in cell averages is achieved through  introducing an adaptive flux correction. Additionally, we adopt a limiting approach so that the positivity of the numerical solutions is enforced at each time step, without compromising accuracy, particularly for smooth solutions. 

The remaining part of the paper is organized as follows. 
In Section 2, we design the numerical method with spatial discretization for one dimensional problems, 
 and show the semi-discrete energy inequality.  
Section 3 is on the time discretization, the free energy dissipation property, the preservation of equilibria, and the positivity-preserving property of the fully discretized scheme under a local flux correction.  
In Section 4 we outline the limiting process,  the hybrid algorithm, and a brief discussion on strong-stability-preserving (SSP) time discretization.
Then in Section 5, we present numerical examples for {one dimensional problems.} Some conclusions are reported in Section 6.

\section{Direct DG discretization in space}	
In this section,  we introduce our DDG scheme for (\ref{IVP}) specifically within one-dimensional space. 
The extension of this formulation to Cartesian meshes in multidimensional scenarios is readily achievable.

\subsection{Scheme formulation}
In the one dimensional setting, where $\Omega = [a, b]$ is a bounded interval,  we employ a mesh partition consisting of a family of $N$ computational cells, denoted as $I_i$'s,  such that each $I_i$ is defined as  $(x_{i-\frac12}, x_{i+\frac12})$, where 
$$
a=x_{\frac12}<x_1<\cdots <x_{N-\frac12}<x_N<x_{N+\frac12}=b,
$$
and the cell center
$x_i =\frac{1}{2}(x_{i-\frac12} + x_{i+\frac12})$. Denote the mesh step $h_i=x_{i+\frac12}-x_{i-\frac12}$ and 
$h=\max_{1\leq i \leq N}h_i$.  

We will seek a numerical solution in the discontinuous piecewise polynomial space
$$
V_h= \left\{\xi(x)\in L^2(\Omega), \quad \xi|_{I_i}\in P^k(I_i), i=1,\cdots, N\right \},
$$
where $P^k(I_i)$ is the space of $k$-th order polynomials on $I_i$. To define the DG method we rewrite the equation \eqref{IVP} 
into the following system 
\begin{subequations}\label{reform}
\begin{numcases}{}
q =V(x)+H'(\rho) +W*\rho,\label{reform_q}\\
\partial_t \rho =\partial_x (\rho \partial_x q), \label{reform_rho}
\end{numcases}
\end{subequations}
where $q$ is the energy flux.
By applying the direct DG approximation, we obtain the following scheme.   We seek $\rho_h(t,\cdot) \in V_h$ such that for any $\eta(x), \xi(x) \in V_h$,
\begin{subequations}\label{semi_DG}
\begin{numcases}{}
\int_{I_i}q_h \eta\,dx =\int_{I_i}\left(V(x) +H'(\rho_h)+\sum_{m=1}^N \int_{I_m}  W(x-y) \rho_h(t, y)\,dy \right) \eta\,dx,\\
\int_{I_i}\partial_t \rho_h\xi \,dx =-\int_{I_i} \rho_h \partial_x q_{h} \partial_x \xi\,dx +\{\rho_h\}\widehat{\partial_x q_h}\xi |_{\partial I_i} +\{\rho_h\}\partial_x \xi (q_h-\{q_h\})|_{\partial I_i}.
\end{numcases}
\end{subequations}
Here $q_h(t, \cdot)$ is a projection of (\ref{reform}a) when evaluated at $\rho_h$,  and  the numerical flux at the interior cell interface $x_{i+\frac12}$ is given by
\begin{equation}
\label{DDGflux}
\widehat{\partial_xq_h}=\beta_0 \frac{[q_h]}{h} +\{\partial_x q_h\}+\beta_1 h[\partial_x^2 q_h].
\end{equation}
For simplicity, the meshsize is taken to be uniform and given by $h$. Otherwise the mesh in the flux formula needs to be replaced by 
$\frac{1}{2}(h_i +h_{i+1})$.   The notations $[q_h]=q_h^+ - q_h^-$ and $\{q_h\}=\frac{q_h^++q_h^-}{2}$ are adopted, with 
 $q_h^-$ and $q_h^+$ being the left and right limit of $q_h$. 

The boundary numerical flux is given to weakly enforce the specified boundary conditions. 
\begin{itemize}
\item[(B1)] If the zero-flux boundary conditions are specified at $x=a, b$, we simply set 
\begin{align*}
& \widehat{\partial_x q_h}=0,  \quad \{q_h\}=q_h^+ \quad   \{\rho_h\}=\rho_h^+ \quad \text{ at } x=x_{\frac12}, \\ 
& \widehat{\partial_x q_h}=0,  \quad \{q_h\}=q_h^- \quad    \{\rho_h\}=\rho_h^- \quad  \text{ at } x= x_{N+\frac12}.
\end{align*}
\item[(B2)] If Dirichlet boundary conditions are given, for instance, $\rho|_{\partial \Omega}=g(t, x)$, we set
\begin{align*}
\widehat{\partial_x q_h} & =\frac{\beta_0}{h } [q_h] +\partial_x q_{h}^+, \quad \{q_h\}=q_h^+-\frac{1}{2}[q_h], \quad   \{\rho_h\}=g(t, a)  \\
\widehat{\partial_x q_h} & =\frac{\beta_0}{h } [q_h] +\partial_x q_{h}^-, \quad \{q_h\}=q_h^-+\frac{1}{2}[q_h], \quad  \{\rho_h\}=g(t, b)
\end{align*} 
 with 
 $$
[q_h]  =
\left\{\begin{array}{ll}
q_h^+- \left( V(a) +H'(g(t, a)) +\sum\limits_{m=1}^N\int_{I_m} W(a-y)\rho_h(t, y)dy \right) &\quad \text{ for } x=a, \\
 \left( V(b) +H'(g(t, b)) +\sum\limits_{m=1}^N\int_{I_m} W(b-y)\rho_h(t, y)dy \right)- q_h^- &\quad \text{ for } x=b. 
\end{array} 
\right.
$$
\end{itemize}
We emphasize the significance of employing various boundary conditions in practical applications, as they play a crucial role in driving systems out of equilibrium and generating non-vanishing ionic fluxes. The numerical method outlined in this study offers flexibility for adaptation to different boundary conditions by adjusting the appropriate boundary fluxes.

\subsection{Energy dissipation inequality} We will demonstrate that the semi-discrete scheme satisfies the following energy dissipation inequality,  
denoted by the discrete energy: 
$$
E_h(t):=\sum_{i=1}^N\int_{I_i} \left(V(x) \rho_h(t ,x)+H(\rho_h(t, x)) +\frac{1}{2}\rho_h(t, x)  \sum_{m=1}^N\int_{I_m} W(x-y)\rho_h(t, y)dy\right)\,dx.
$$ 
Define the energy norm of $q_h$ by 
\begin{align}\label{qh_energy}
\|q_h\|_E:= \left[ \sum_{i=1}^N\int_{I_i} \rho_h |\partial_xq_h|^2\,dx 
    + \sum_{i=1}^{N-1}\left.\left(\frac{1}{h }\{\rho_h\}[q_h]^2 \right)\right|_{x_{i+\frac12}} \right]^{\frac12}.
\end{align}

\begin{thm}\label{semiEnergyDis} 
Assuming that the semi-discrete scheme with (B1) boundary setup admits a positive solution $\rho_h$, 
then it satisfies an energy dissipation law expressed as   
 \begin{align}\label{dE}
\frac{d}{dt} E_h(t)\leq -\gamma \|q_h\|_E^2
\end{align}
for some $\gamma \in (0, 1)$, provided 
 \begin{align}\label{lb}
  \beta_0 \geq  \max_{1\leq i \leq N-1}\frac{\left.\{\rho_h\}  \left( \{ \partial_x q_h \} +\frac{\beta_1}{2} h [\partial_x^2 q_h]\right)^2\right|_{x_{i+\frac12}} }{\frac{1}{2h }\left(\int_{I_i}+\int_{I_{i+1}}\right)\rho_h |\partial_x q_h |^2dx}.
\end{align}
\end{thm}
\begin{proof}
Summing \eqref{semi_DG} over the index $i$'s, we obtain a global formulation
\begin{subequations}\label{sum_semi}
\begin{align}
\sum_{i=1}^N\int_{I_i} q_h\eta\,dx &= \sum_{i=1}^N\int_{I_i} \left(V(x) +H'(\rho_h) +\sum_{m=1}^N\int_{I_m} W(x-y)\rho_h(t,y)\,dy \right)\eta\,dx, \label{q_proj}\\
\sum_{i=1}^N\int_{I_i} \partial_t\rho_h \xi\, dx &= - \sum_{i=1}^N \int_{I_i} \rho_h \partial_xq_h \partial_x\xi\,dx - \sum_{i=1}^{N-1} \{\rho_h \} \left.\left(
\widehat{\partial_xq_h} [\xi]+ \{ \partial_x\xi\}[q_h]\right)\right|_{x_{i+\frac12}},
\end{align}
\end{subequations}
where boundary fluxes given in (B1) have been used. Take $\eta=\partial_t\rho_h$  in (\ref{sum_semi}a) to obtain
\begin{align*}
\sum_{i=1}^N\int_{I_i}\partial_t\rho_h q_h\,dx & =\sum_{i=1}^N\int_{I_i} \left( V(x)+H'(\rho_h)+\sum_{m=1}^N\int_{I_m}W(x-y)\rho_h(t,y)\,dy \right)\partial_t\rho_h\,dx\\
& =\frac{d}{dt}E_h(t),
\end{align*}
where we have used the symmetry of $W$.  This when combined with  (\ref{sum_semi}b)  taking $\xi=q_h$ in (\ref{sum_semi}b) gives 
\begin{align*}
\frac{d}{dt}E_h(t) &= -\sum_{i=1}^N\int_{I_i} \rho_h |\partial_xq_h|^2\,dx -\sum_{i=1}^{N-1} \{\rho_h \} [q_h]\left.\left( \widehat{\partial_xq_h}+\{\partial_xq_h\}\right)\right|_{x_{i+\frac12}}\\
 & =   -\sum_{i=1}^N\int_{I_i} \rho_h|\partial_xq_h|^2\,dx -\sum_{i=1}^{N-1} \{\rho_h \} \left.\left(\frac{\beta_0}{h }[q_h]^2 +[q_h]( 2\{\partial_xq_h\}
 +\beta_1h [\partial_x^2q_h])\right)\right|_{x_{i+\frac12}}.
\end{align*}
Using Young's inequality, we obtain
$$
-[q_h](2\{\partial_xq_h\} + \beta_1h [\partial_x^2q_h])\leq \frac{\beta_0(1-\gamma)}{h }[q_h]^2 + \frac{h }{4\beta_0(1-\gamma)} \left( 2\{\partial_xq_h\} 
    +\beta_1h [\partial_x^2q_h]\right)^2
$$
for any $\gamma \in (0,1)$. It follows that
\begin{align*}
\frac{d}{dt}E_h(t) & \leq  -\gamma\left[ \sum_{i=1}^N\int_{I_i} \rho_h |\partial_xq_h|^2\,dx 
  + \sum_{i=1}^{N-1}\left.\left(\frac{\beta_0}{h }\{\rho_h\}[q_h]^2 \right)\right|_{x_{i+\frac12}} \right] \\ 
& \quad - \left[ (1-\gamma)\sum_{i=1}^N\int_{I_i} \rho_h |\partial_xq_h|^2\,dx 
  - \sum_{i=1}^{N-1} \left.\frac{h \{\rho_h\}}{4\beta_0(1-\gamma)}\left( 2\{\partial_xq_h\} +\beta_1h [\partial_x^2q_h]\right)^2\right|_{x_{i+\frac12}} \right]. 
\end{align*}
For any fixed $\beta_1$, it suffices to choose $\beta_0>1$ large enough so that 
\[
\frac{h \{\rho_h\}}{4\beta_0(1-\gamma)}\left( 2\{\partial_xq_h\} +\beta_1h [\partial_x^2q_h]\right)^2
\leq \frac{1-\gamma}{2}\left(\int_{I_i} + \int_{I_{i+1}}\right)\rho_h|\partial_xq_h|^2\,dx,
\]
at all the interface points $x_{i+\frac12}, i=1, \ldots, N-1$.  That is, 
\begin{equation}\label{semi_beta_0}
\beta_0 \geq \frac{1}{4(1-\gamma)^2}\frac{\left.\{\rho_h\}\left( 2\{\partial_xq_h\} +\beta_1h [\partial_x^2q_h]\right)^2\right|_{x_{i+\frac12}}}
{\frac{1}{2h }\left(\int_{I_i} + \int_{I_{i+1}}\right)\rho_h|\partial_xq_h|^2\,dx}. 
\end{equation}
For $\beta_0$ satisfying (\ref{lb}), we can indeed find $\gamma \in (0, 1)$ such that (\ref{semi_beta_0}) holds, hence (\ref{dE}).
\end{proof}

\begin{rem} The lower bound for $\beta_0$ given in (\ref{lb}), based on the evaluation of the ratio over numerical solutions, proves to be cumbersome to use.  To establish a range for $\beta_0$ independent of numerical solutions, we assume that $\rho_h$ represents a high-order approximation of a positive and smooth $\rho(t,x)$. Without loss of generality, let the accuracy order be $k+1$.  Consequently,  
$$
    \rho_h^\pm(t, x_{i+\frac12}) = \rho(t, x_{i+\frac12})+\mathcal{O}\big(h^{k+1}\big) \text{ and }\rho_h(t, x) = \rho(t, x_{i+\frac12})+\mathcal{O}(h )
\text{ for } j = 1, \ldots, N-1.
$$
For sufficiently small $h $, we have 
$$
\frac{\{\rho_h\}|_{x_{i+\frac12}}}{\rho_h|_{I_i}},\quad 
\frac{\{\rho_h\}|_{x_{i+\frac12}}}{\rho_h|_{I_i+1}}\leq 2.
$$  
Therefore, we can infer that
\begin{align*}
\frac{\left.\{\rho_h\}\left( 2\{\partial_xq_h\} +\beta_1h [\partial_x^2q_h]\right)^2\right|_{x_{i+\frac12}}}
{\frac{1}{2h }\left(\int_{I_i} + \int_{I_{i+1}}\right)\rho_h|\partial_xq_h|^2\,dx}
\leq 2\frac{\left.\left( \{ \partial_x q_h \} +\frac{\beta_1}{2} h [\partial_x^2 q_h]\right)^2\right|_{x_{i+\frac12}} }
{\frac{1}{2h }\left(\int_{I_i}+\int_{I_{i+1}}\right) |\partial_x q_h |^2dx}
\leq 2\Gamma,
\end{align*}
where 
\[
\Gamma = \max_{1\leq i\leq N-1} \frac{\left.\left(\{\partial_x q_h\}+\frac{\beta_1}{2} h [\partial_x^2 q_h]\right)^2\right|_{x_{i+\frac12}} }
{\frac{1}{2h }\left(\int_{I_i}+\int_{I_{i+1}}\right) |\partial_x q_h |^2dx}.
\]
The estimate of $\Gamma$ as provided in  \cite{Liu15} enables us to conclude that it suffices to choose   
\begin{equation}
\beta_0 > 2 \Gamma  = 2 k^2 \left(1 - \beta_1 (k^2-1)  +\frac{\beta_1^2}{3} (k^2-1)^2 \right).
\end{equation}
\end{rem}

\section{Time discretization and structure preservation}
For the time discretization, let $\Delta t$ be the time step, and $\rho_h^n(x)$ denote the numerical approximation to $\rho_h(t^n, x)$ with $t^n = n\Delta t$.

\subsection{The energy dissipation law} 
We apply the Forward Euler method for the time discretization of (\ref{semi_DG}). We seek   $\rho_h^{n+1}(x)\in V_h$ such that for any $\xi(x), \eta(x) \in V_h$, the following conditions hold: 
\begin{subequations}\label{fully_DG}
\begin{numcases}{}
\int_{I_i}q_h^n \eta\,dx =\int_{I_i}\left(V(x) +H'(\rho_h^n)+\sum_{m=1}^N \int_{I_m}  W(x-y) \rho_h^n(y)\,dy \right) \eta\,dx,\\
    \int_{I_i}D_t\rho_h^n \xi \,dx =-\int_{I_i} \rho_h^n \partial_x q_{h}^n \partial_x \xi\,dx 
	+\{\rho_h^n\}\left.\left[\widehat{\partial_x q_h^n }\xi + \partial_x \xi (q_h^n-\{q_h^n\})\right]\right|_{\partial I_i}.
\end{numcases}
\end{subequations}
Here and in what follows, $D_t$ is sued to denote the forward difference operator in time, defined as 
$$
D_t u^n=\frac{u^{n+1}-u^n}{\Delta t}
$$
for any function $u^n$.  The explicit time discretization is straightforward to implement, while still preserving the energy dissipation law  under certain constraints on the time step.  

\begin{thm}\label{thfullyenergy}
With the discrete energy defined as
$$
E^n =\sum_{i=1}^N\int_{I_i} \left(V(x) \rho_h^n(x)+H(\rho_h^n(x)) +\frac{1}{2}\rho_h^n(x)  \sum_{m=1}^N\int_{I_m} W(x-y)\rho_h^n(y)dy\right)dx,
$$ 
the DG scheme \eqref{fully_DG},  subject to zero-flux boundary setup (B1),  satisfies
$$
    D_tE^n\leq -\frac{\gamma}{2}\|q_h^n\|_E^2
$$
for some $\gamma \in (0, 1)$, provided $\rho_h^n(x)$ remains positive, $\Delta t$ is suitably small, and 
 \begin{align}\label{lb+}
  \beta_0 >  \max_{1\leq i \leq N-1}\frac{\left.\{\rho_h^n\}  \left( \{ \partial_x q_h^n \} +\frac{\beta_1}{2} h [\partial_x^2 q_h^n]\right)^2\right|_{x_{i+\frac12}} }{\frac{1}{2h }\left(\int_{I_i}+\int_{I_{i+1}}\right)\rho_h^n |\partial_x q_h^n |^2dx}.
\end{align}
\end{thm}

\begin{proof}
Summing equation \eqref{fully_DG} over all indices  $i$, we obtain a global formulation:
\begin{align}
\sum_{i=1}^N\int_{I_i} q_h^n\eta\,dx  
& = \sum_{i=1}^N\int_{I_i}\left(V(x) +H'(\rho_h^n) +\sum_{m=1}^N\int_{I_m} W(x-y)\rho_h^n(y)\,dy \right)\eta\,dx,\label{q_sum}\\
\sum_{i=1}^N\int_{I_i} D_t \rho_h^n  \xi\, dx & = - \sum_{i=1}^N \int_{I_i} \rho_h^n  \partial_xq_h^n \partial_x\xi\,dx 
	- \sum_{i=1}^{N-1} \{\rho_h^n \} \left.\left(\widehat{\partial_xq^n_h} [\xi]+ \{ \partial_x\xi\}[q^n_h]\right)\right|_{x_{i+\frac12}}. \label{rho_sum}
\end{align}
Take $\eta=D_t\rho_h^n $  in (\ref{q_sum}) to obtain
\begin{align*}
 \int_\Omega D_t\rho_h^n q_h^n dx  & =\int_\Omega \left( V(x)+H'(\rho_h^n(x))   +  \sum_{m=1}^N\int_{I_m}W(x-y)\rho_h^n(y)\,dy \right)D_t\rho_h^n\,dx \\
& = D_t E^n -\frac{\Delta t}{2} \int_{\Omega}H''(\cdot)(D_t\rho_h^n)^2dx  \\
& \,\,\,\,\,\, - \frac{\Delta t}{2}\int_{\Omega}\int_\Omega W(x-y)D_t\rho_h^n(x)D_t\rho_h^n(y)dxdy,
\end{align*}
where $\cdot$ is a mean value between $\rho_h^n$ and $\rho_h^{n+1}$. Taking $\xi=q_h^n$, (\ref{rho_sum}) becomes
\begin{align*}
 \int_\Omega D_t\rho_h^n q_h^n dx  &= -\sum_{i=1}^N\int_{I_i} \rho_h^n |\partial_xq_h^n |^2\,dx -\sum_{i=1}^{N-1} \{\rho_h^n \} [q_h^n]\left.\left( \widehat{\partial_xq_h^n}+\{\partial_xq_h^n\}\right)\right|_{x_{i+\frac12}}\\
& \leq  -\gamma \|q_h^n\|_E^2,
\end{align*}
provided $\beta_0$ is suitably large so that 
$$
\beta_0 >\beta_0(1-\gamma)^2 \geq  
    \max_{1\leq i \leq N-1}\left.\frac{\{\rho_h^n\}  \left( \{ \partial_x q_h^n \} +\frac{\beta_1}{2} h [\partial_x^2 q_h^n]\right)^2 }{\frac{1}{2h }\left(\int_{I_i}+\int_{I_{i+1}}\right)\rho_h^n |\partial_x q_h^n |^2dx}\right|_{x_{i+\frac12}}.
$$
Hence, we deduce that 
$$
D_t E^n\leq -\gamma \|q_h^n\|_E^2 + \frac{\Delta t}{2} (Q_1+Q_2), 
$$
with 
\begin{align*}
Q_1 & = \int_{\Omega}H''(\cdot)(D_t\rho_h^n)^2\,dx, \\
Q_2 & = \int_\Omega\int_\Omega W(x-y)D_t\rho_h^n(x)D_t\rho_h^n(y)\,dxdy.
\end{align*}
The claimed estimate follows if 
$$
\Delta t \leq \frac{\gamma \|q_h^n\|_E^2}{(Q_1+Q_2)^+}. 
$$
\end{proof}

\begin{rem}
Notice that there might be a requirement on $W$ such that $Q_2 \geq 0$, as discussed in \cite{BCH2020}. However, we prefer to state the result in its most general form.  
\end{rem}
Furthermore, the fully discrete scheme is capable of preserving positive equilibrium solutions.

\begin{thm}
Suppose that the fully-discrete scheme \eqref{fully_DG} with (B1) boundary setup yields a nonnegative solution $\rho_h^n$. Assume that an equilibrium solution, denoted by $\rho_\infty$,  satisfies  $\frac{d}{dt}E[\rho_\infty] = 0$ and is positive in $\Omega$.
Let $\mathcal{P}_{V_h}$ be the projection operator onto the solution space $V_h$. Then it preserves the equilibrium solution in the following sense: 
\begin{enumerate}
\item[(i)] If $\rho_h^0(x) = \mathcal{P}_{V_h}\rho_\infty(x)$, then $\rho_h^1(x) = \rho_h^0(x)$ in $\Omega$.
\item[(ii)] If $D_tE^n \equiv 0$, then $\rho_h^{n+1} = \rho_h^n \quad \forall x\in \Omega$.
\end{enumerate}
\end{thm}
\begin{proof}
    (i) 
Let $q_\infty =V + H'(\rho_\infty) + W*\rho_\infty$.
Since $\rho_\infty$ is a positive equilibrium that satisfies 
\[
\frac{d}{dt}E[\rho_\infty] = -\int_\Omega \rho_\infty|\nabla_x(V + H'(\rho_\infty) + W*\rho_\infty)|^2\,dx
= -\int_\Omega \rho_\infty |\nabla_x q_\infty|^2\,dx = 0, 
\]
then $q_\infty = V + H'(\rho_\infty) + W*\rho_\infty$ must be a constant in $\Omega$. As a result, the $L^2$ projection of $q_\infty$ onto the solution space $V_h$, denoted by $\mathcal{P}_{V_h} q_\infty$, is constant  as well. Now we start with the initial data $\rho_h^0(x) = \mathcal{P}_{V_h}\rho_\infty(x)$, and then $q_h^0(x) = \mathcal{P}_{V_h} q_\infty$ is a constant.
Using the scheme (\ref{fully_DG}b), we have
\[
    \int_{I_i}D_t\rho_h^0\xi\,dx = 0 \quad \text{ for any } \xi(x)\in V_h, i = 1, \ldots, N. 
\]
    Note that $D_t\rho_h^0 = \frac{\rho_h^1-\rho_h^0}{\Delta t} \in V_h$ can be taken as $\xi(x)$, we thus have 
 \[
     D_t\rho_h^0 = \frac{\rho_h^1-\rho_h^0}{\Delta t} = 0 \quad \text{ for } i = 1, \ldots, N,
 \]
 i.e., $\rho_h^1(x) = \rho_h^0(x)$ in $\Omega$.

(ii) Suppose that $D_t E^n \equiv 0$. 
    Since $D_tE^n\leq -\frac{\gamma}{2}\|q_h^n\|_E^2$, it implies that
    \[
        \|q_h^n\|_E^2 = \sum_{i=1}^N\int_{I_i} \rho_h^n |\partial_xq_h^n|^2\,dx 
    + \sum_{i=1}^{N-1}\left.\left(\frac{1}{h }\{\rho_h^n\}[q_h^n]^2 \right)\right|_{x_{i+\frac12}} = 0.
    \]
Thanks to the nonnegativity of $\rho_h^n$ and that $\rho_h^n$ being a smooth polynomial function in each cell $I_i$, we must have
    \[
        \partial_x q_h^n \equiv 0 \text{ in each } I_i, \text{ and } [q_h^n] = 0 \text{ for } i = 1, \ldots, N-1.
    \]
Therefore $q_h^n$ is constant in the whole domain $\Omega$. 
    Again using the scheme (\ref{fully_DG}b), we have
    \[
        \int_{I_i}D_t\rho_h^n\xi\,dx = 0 \quad \text{ for any } \xi(x)\in V_h, i = 1, \ldots, N. 
        \Longrightarrow \rho_h^{n+1} = \rho_h^n \text{ in } \Omega.
    \]
\end{proof}

\subsection{Preservation of positivity of cell averages}
The numerical method \eqref{fully_DG} itself does not guarantee the positivity $\rho_h^n$ as time evolves. The approach outlined in \cite{LW17} aims to reconstruct positive densities based on cell averages.  However, as demonstrated in \cite{LW17} numerical evidence suggests that even the cell average $\bar\rho_i^{n}$ can become nonpositive for certain $i$ as $n$ increases,  leading to the failure of the reconstruction step. In this paper, we propose a modification to ensure  the positivity of $\bar\rho_i^{n}$  at each time step. We introduce an additional local term in the diffusive numerical flux to facilitate the recovery of positivity for cell averages.  

Firstly,  let us  consider the Euler forward time stepping. The modified DG scheme now becomes 
\begin{subequations}\label{fullyPDG}
\begin{numcases}{}
\int_{I_i}q_h^n \eta\,dx =\int_{I_i}\left(V(x) +H'(\rho_h^n)+\sum_{m=1}^N \int_{I_m}  W(x-y) \rho_h^n(t, y)\,dy \right) \eta\,dx,\\
\int_{I_i}D_t \rho_h^n\xi \,dx =-\int_{I_i} \rho_h^n \partial_x q_{h}^n \partial_x \xi\,dx +\{\rho_h^n\}\widetilde{\partial_x q_h^n}\xi |_{\partial I_i} +\{\rho_h^n\}\partial_x \xi (q_h^n-\{q_h^n\})|_{\partial I_i},
\end{numcases}
\end{subequations}
where the flux at the interior cell interface $x_{i+\frac12}$ is given by
\begin{equation}\label{flux}
\widetilde{\partial_xq_h}=\widehat{\partial_xq_h} + \frac{1}{2} \beta_{i+\frac12}[\rho_h], 
\end{equation}
where 
\begin{equation}\label{defbeta12}
    \beta_{i+\frac12}= \left\{\begin{array}{cl}
        \frac{|\widehat{\partial_x q_h^n}|}{\{\rho_h^n\}}, &\text{ if } \{\rho_h^n\} > 0, \\
0, &\text{ if } \{\rho_h^n\} = 0.
\end{array}\right.
\end{equation}
This scheme  ensures that the cell average $\bar\rho_i^{n} > 0$ for all $n$, as demonstrated  by the following  theorem.
\begin{thm}\label{THMfullyPDG}
For the fully discrete scheme \eqref{fullyPDG}, the cell average $\bar\rho_i^{n+1} > 0$ is guaranteed,  provided $\rho_h^n(x)>0$ and the CFL condition
\[
    \lambda :=\frac{\Delta t}{h ^2} \leq \frac{\omega_1}{h }\min_{1\leq i\leq N-1} \left.|\widehat{\partial_x q_h^n}|^{-1}\right|_{x_{i+\frac12}}
\]
is satisfied. Here,  $\omega_1$ represents the weight of the Gauss-Lobatto quadrature rules with $M \geq \frac{k+3}{2}$ points.
\end{thm}
\begin{proof} 
Take $\xi = 1/h $ in (\ref{fullyPDG}b), we have 
\begin{align*}
    \bar\rho_i^{n+1} =& \bar\rho_i^n + \left.\lambda\{\rho_h^n\}h \left(\widehat{\partial_x q_h^n} + \frac{\beta_{i+\frac12}}{2}[\rho_h^n]\right)\right|_{x_{i+\frac12}}
    -\left.\lambda h \{\rho_h^n\}\left(\widehat{\partial_x q_h^n} + \frac{\beta_{i-\frac12}}{2}[\rho_h^n]\right)\right|_{x_{i-\frac12}}.
\end{align*}  
Note that $\rho^n$ is a polynomial of degree $k$, the Gauss-Lobatto quadrature with $M \geq \frac{k+3}{2}$ is exact for evaluating the cell average $ \bar \rho_i^n$.  
More specifically, we have 
$$
 \bar\rho_i^n=\sum_{m=1}^{M} \omega_m\rho_h^n(\hat x_i^m)
$$
with $\hat x_i^1 = x_{i-\frac12}, \quad \hat x_i^{M} = x_{i+\frac12}$, and $\omega_1=\omega_M$ due to the symmetry of the Gaussian quadrature.
Hence 
    \begin{align*}
    \bar\rho_i^{n+1} =& \sum_{m=1}^{M} \omega_m\rho_h^n(\hat x_i^m) + \frac{\lambda}{2}h \left(\widehat{\partial_xq_h^n} + \beta_{i+\frac12}\{\rho_h^n\}\right)\rho_h^n(x_{i+\frac12}^+)\\
   & \hspace{2.6cm} +\frac{\lambda}{2}h \left(\widehat{\partial_xq_h^n} - \beta_{i+\frac12}\{\rho_h^n\}\right)\rho_h^n(x_{i+\frac12}^-)\\
&\hspace{2.6cm} -\frac{\lambda}{2}h \left(\widehat{\partial_xq_h^n} + \beta_{i-\frac12}\{\rho_h^n\}\right)\rho_h^n(x_{i-\frac12}^+)\\
& \hspace{2.6cm} +\frac{\lambda}{2}h \left(-\widehat{\partial_xq_h^n} + \beta_{i-\frac12}\{\rho_h^n\}\right)\rho_h^n(x_{i-\frac12}^-)\\
=& \sum_{m=2}^{M-1} \omega_m\rho_h^n(\hat x_i^m) 
+ \frac{\lambda}{2}h \left( \widehat{\partial_x q_h^n} + \beta_{i+\frac12}\{\rho_h^n\}\right)\rho_h^n(x_{i+\frac12}^+)\\
&\hspace{2.6cm}+ \frac{\lambda}{2}h \left(-\widehat{\partial_xq_h^n} + \beta_{i-\frac12}\{\rho_h^n\}\right)\rho_h^n(x_{i-\frac12}^-)\\
&\hspace{2.6cm}+ \left[\omega_1 - \frac{\lambda}{2}h (\widehat{\partial_x q_h^n} + \beta_{i-\frac12}\{\rho_h^n\})\right]\rho_h^n(x_{i-\frac12}^+)\\
&\hspace{2.6cm}+ \left[\omega_M - \frac{\lambda}{2}h (-\widehat{\partial_x q_h^n} + \beta_{i+\frac12}\{\rho_h^n\})\right]\rho_h^n(x_{i+\frac12}^-).
\end{align*}
Since the solution values $\rho_h^n(\hat x_i^m)\geq0$ and the weights $\omega_m \geq 0$, the first term is non-negative. The positivity of the second and third terms 
is ensured thank to the choice of $\beta_{i+\frac12}$ in \eqref{defbeta12}, 
\[
    \widehat{\partial_xq_h^n} + \beta_{i+\frac12}\{\rho_h^n\} \geq 0, \quad -\widehat{\partial_xq_h^n} + \beta_{i-\frac12}\{\rho_h^n\} \geq 0.
\]
The positivity of the last two terms is guaranteed  provided
\[
    \lambda \leq \frac{\omega_1}{h }\min_{1\leq i\leq N-1} \left.|\widehat{\partial_x q_h^n}|^{-1}\right|_{x_{i+\frac12}}.
\]
\end{proof}

\begin{rem} We argue that the effects of  $\widetilde{\partial_x q_h}$ on the energy dissipation are under control.  With the local flux correction, the energy dissipation estimate in Theorem \ref{thfullyenergy}
now reads: 
$$
    D_tE^n\leq -\frac{\gamma}{2}\|q_h^n\|_E^2 -\frac12\sum_{i=1}^{N-1} [\rho_h^n][q_h^n]\left.|\widehat{\partial_xq_h^n}|\right|_{x_{i+\frac12}}.
$$
For a convergent DG scheme of $k+1$-th order, one expects to have $[\rho_h]\sim \mathcal{O}\big(h^{k+1}\big)$. This, combined with careful estimates as in the proof of Theorem \ref{semiEnergyDis}, shows that the extra term is bounded from above by: 
$$
C\beta_0h^{k+1}\|q_h^n\|^2_E.
$$
Hence, for sufficiently small $h$ we still have the energy dissipation law  
$$
    D_tE^n\leq -\frac{\gamma}{4}\|q_h^n\|_E^2.
    $$
\end{rem}

\section{The hybrid algorithm with limiting}

\subsection{A positivity-preserving limiter} \label{sec_limiter}
The discrete energy dissipation law requires that $\rho_h(t,x)$ in the semi-discrete setting or $\rho_h^n(x)$ in the fully discrete setting be positive for any $x$. 
We enforce the solution positivity through a positivity-preserving limiter,  based on the one introduced in \cite{ZS10} for scalar hyperbolic conservation laws. As this step is independent of   time $t$, we omit the time dependence and denote  the DG solution as $\rho_h(x)$.
Assuming the cell average 
\[
\bar{\rho}_i= \frac{1}{|I_i|}\int_{I_i}\rho_h(x)\,dx \geq \delta>0 
\]  
for a small parameter $\delta$,  we construct another polynomial $\rho_h^{\delta}(x) \in V_h$ 
using the positivity-preserving limiter
\begin{equation}\label{reconstruct}
\rho_h^{\delta}(x)|_{I_i}:= \bar{\rho}_i+\frac{\bar{\rho}_i-\delta}{\bar{\rho}_i-\min_{I_i} \rho_h(x)} (\rho_h(x)-\bar{\rho}_i)
\quad  \text{ if } \min_{I_i} \rho_h(x)<\delta.
\end{equation}
This modification maintains the cell average $\bar{\rho}_i$ and ensures  
$$
\min_{x\in I_i} \rho^\delta_h(x)\geq\delta.
$$
Moreover, this limiting process preserves the order of accuracy  for suitably small $\delta$.  
For completeness, we recall the following result.
\begin{lem}(\cite{LW17}) If $\bar{\rho}_j>\delta$,  then the reconstruction satisfies the estimate
$$
|\rho_h^{\delta}(x)-\rho_h(x)|\leq C(k) \left( ||\rho_h(x)-\rho(x)||_\infty+ \delta\right),\quad \forall x\in I_j,
$$
where $C(k)$ is a constant depending on $k$.  This implies that the reconstructed $\rho_h^{\delta}(x)$ in \eqref{reconstruct}
    does not destroy the order of accuracy  when $\delta < h ^{k+1}$.
\end{lem}
We refer to \cite{ZS10} for further details on this type of positivity-preserving limiters.  

\subsection{The hybrid algorithm}
Here we state a hybrid algorithm that preserves the energy dissipation and positivity of $\rho$.
\begin{itemize}
\item Initialization for $\rho_h^0$. Project $\rho_0(x)$ onto $V_h$ to obtain $\rho_h^0$.
\item Time evolution.\\
For $n = 0, 1, 2, \ldots $

\begin{enumerate}
    \item[(i)] Perform the positivity-preserving limiter \eqref{reconstruct} on $\rho_h^n$
    and set $\rho_h^n = \rho_h^{\delta}$ if needed.
        Using $\rho_h^{n}$, solve (\ref{fully_DG}) to obtain $q_h^{n}$ and then $\rho_h^{n+1}$.

    \item[(ii)] Check the positivity of the cell average $\bar\rho_i^{n+1}$.
\begin{enumerate}
\item[] If $\bar\rho_i^{n+1} > 0$, then

\hspace{0.5cm} Continue.
\item[] Else

\hspace{0.5cm} Using $\rho_h^{n}$, solve \eqref{fullyPDG} to obtain $q_h^n$ and then $\rho_h^{n+1}$.
\item[] End If
        \end{enumerate}
\end{enumerate}\mbox{}\\
End For
\end{itemize}

\subsection{SSP time discretization} 
The time discretization in schemes \eqref{fully_DG} and \eqref{fullyPDG} can be implemented using explicit high order Runge-Kutta methods.   
Indeed, general results on forward Euler time discretization can be extended to high order strong-stability-preserving (SSP), also known as total variation diminishing (TVD), Runge-Kutta time discretization \cite{SO88, GST2001}, which is commonly employed in practice to ensure stability and enhance temporal accuracy. 
In our numerical simulation, we use the third order explicit Runge-Kutta (RK3) method for time discretization to solve the ODE system of the form 
$\frac{d \textbf{a}}{dt}=\mathcal{L}(\textbf{a})$:
\begin{align}
	\textbf{a}^{(1)} &= \textbf{a}^n + \Delta t \mathcal{L}(\textbf{a}^n), \nonumber \\
	\textbf{a}^{(2)} &= \frac{3}{4}\textbf{a}^n + \frac{1}{4}\textbf{a}^{(1)} + \frac{1}{4}\Delta t \mathcal{L}(\textbf{a}^{(1)}), \label{eq:RK3} \\
	\textbf{a}^{n+1} &= \frac{1}{3}\textbf{a}^n + \frac{2}{3}\textbf{a}^{(2)} + \frac{2}{3}\Delta t \mathcal{L}(\textbf{a}^{(2)}). \nonumber
\end{align} 
The primary advantage of this method lies in its ability to compute each stage successively, thereby  significantly reducing computational costs and memory overhead.  Additionally, employing DG schemes with time evolution via the strong stability preserving (SSP) Runge–Kutta method ensures that  the free energy does not increase at each time step, provided the time step is suitably small. For a proof of such property with RK2 methods, please refer to \cite{LW17} 

\section{Numerical examples}
In this section, we provide numerical examples to assess the effectiveness of our DG schemes.  Initially, we analyze the order of accuracy through   numerical convergence tests. Subsequently, we  present additional examples to evaluate solution properties such as energy-dissipation and positivity preservation, along with the scheme's capacity to capture steady states. The common parameters used for these numerical examples are specified as follows: 
\[
\beta_0 = 5.434, \qquad \beta_1 = 0.15, \qquad \delta = 10^{-12}.
\]

\noindent{\bf Example 1 (Accuracy test)}
Consider the heat equation with exact solutions available:
\begin{align}
    \left\{\begin{array}{ccll}
        \partial_t\rho & = &\partial_x^2\rho = \partial_x\left(\rho\partial_x\left(H'(\rho)\right)\right), & x\in [-\pi,\pi],\\
        \rho(0, x) &=& 2+\sin(x), &x \in[-\pi,\pi],
    \end{array}\right.
\end{align}
with periodic boundary conditions, where $H(\rho) = \rho(\ln \rho -1 )$.
The exact solution and equilibrium are given by
\[
\rho(t, x) = 2 + e^{-t}\sin x, \qquad \rho_{\infty}(x) = 2.
\]
Table \ref{Eg52} below illustrates the optimal order of accuracy of the energy-satisfying DG method with various polynomial degrees $k$ at final time $t =0.1$.
\begin{table}
\centering
\begin{tabular}{|c|c|c|c|c|c|}
\hline
                     &$N_x$       & $16$      & $32$      & $64$      & $128$ \\\hline
\multirow{2}*{$k=1$} &$L^2$ error & 9.257E-03 & 2.359E-03 & 5.932E-04 & 1.485E-04\\\cline{2-6}
                     &$L^2$ order & -         & 1.972     & 1.992     & 1.998 \\\hline
\multirow{2}*{$k=2$} &$L^2$ error & 3.855E-04 & 4.832E-05 & 6.044E-06 & 7.561E-07\\\cline{2-6}
                     &$L^2$ order & -         & 2.996     & 2.999     & 2.999 \\\hline
\multirow{2}*{$k=3$} &$L^2$ error & 6.821E-04 & 3.032E-05 & 1.451E-06 & 8.407E-08\\\cline{2-6}
                     &$L^2$ order & -         & 4.492     & 4.385     & 4.109 \\\hline
\multirow{2}*{$k=4$} &$L^2$ error & 9.754E-04 & 5.604E-05 & 2.362E-06 & 9.200E-08\\\cline{2-6}
                     &$L^2$ order & -         & 4.121     & 4.568     & 4.682 \\\hline
\end{tabular}
\caption{Accuracy test on the heat equation in {\bf Example 1}.}
\label{Eg52}
\end{table}
Notably, for $k=1, 2, 3$, the numerical solutions converge to the exact solution with optimal accuracy.
However, for $k =4$, there exists a slight degeneracy compared to the optimal order of $5$.
This discrepancy may arise from the inherent smoothness of the solution, where higher-degree polynomial approximations can induce artificial oscillations. Furthermore, the comparable $L^2$ error between $k = 3$ and $4$ suggests that increasing the degree of basis polynomials does not  necessarily enhance the approximation accuracy. \\

\noindent{\bf Example 2 (Porous medium equation) }\\
Consider the model problem defined by: 
\begin{align}
    \left\{\begin{array}{ccll}
        \partial_t\rho & = &\partial_x\left(\rho\partial_x\left(2\rho + \frac{x^2}{2}\right)\right), & x\in [-2,2],\\
        \rho(x,0) &=& \max\{1-|x|, 0\}, &x \in[-2,2],
    \end{array}\right.
\end{align}
with periodic boundary conditions.  This corresponds to (\ref{IVP}a) with $V(x) = \frac{x^2}{2}, H(\rho) = \rho^2, W(x)\equiv 0$.
The corresponding steady state is given by: 
\begin{equation}\label{porousmedium_infty}
        \rho_\infty(x) = \max\left\{\left(\frac{3}{8}\right)^{\frac23} - \frac{x^2}{4}, 0\right\}.
\end{equation}
\begin{figure}
\centering
\begin{subfigure}[b]{0.45\textwidth}
\includegraphics[width=\textwidth]{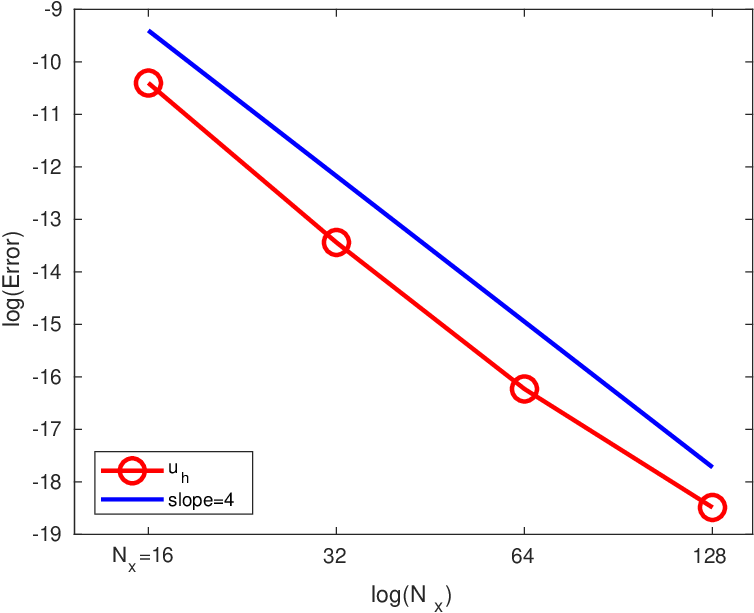}
\caption{$L^2$ error for $k = 3$.}
\label{Eg01error}
\end{subfigure}\quad
\begin{subfigure}[b]{0.45\textwidth}
\includegraphics[width=\textwidth]{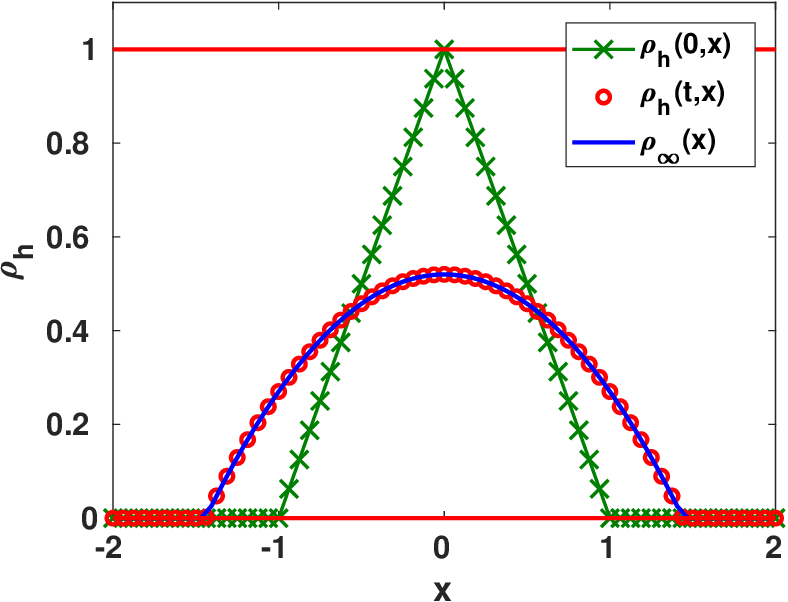}
\caption{$\rho_h$.}
\label{Eg01uh}
\end{subfigure}
\caption{ 
(A) displays the $L^2$ error of the porous media equation in {\bf Example 2} for $k=3$. 
(B) illustrates the numerical solution $\rho_h$ with $N = 64, k=3$ at $t=0$ and $32$, represented by green crosses and red circles, respectively.
The equilibrium $\rho_\infty$ specified in \eqref{porousmedium_infty} is shown by the blue solid curve.
}
\label{figEg01}
\end{figure}
Figure \ref{figEg01}(A) displays the $L^2$ error of the porous medium equation for $k=3$ with different meshes, achieving an optimal accuracy of order $4$.
Figure \ref{figEg01}(B) illustrates the initial data $\rho_h(0,x)$ (green crosses), and the numerical solutions $\rho_h(t=32,x)$ (red circles). 
One can observe that $\rho_h$ quickly converges to the equilibrium $\rho_\infty$ (blue solid curve).\\

\noindent{\bf Example 3 (Attractive-repulsive kernels) }\\
Consider equation (\ref{IVP}) in one dimension with  $H(\rho)=0, V (x) = 0$,  and  the interaction kernel 
$W(x) = \frac{|x|^2}{2}- \ln |x|$, subject to zero flux boundary conditions. 
The corresponding unit-mass steady state is given by
\begin{equation}\label{AR_infty}
\rho_\infty=\frac{1}{\pi}\sqrt{2-x^2}1_{|x|\leq \sqrt{2}},
\end{equation}
and is H\"{o}lder continuous with exponent $\alpha=\frac12$. 
This steady state serves as the unique global minimizer of the free energy $E$, and it is approached by the solutions of (\ref{IVP}) with an exponential convergence rate \cite{CCH14}. 

We compute $\rho_\infty$ by numerically solving (\ref{IVP}) at large time, with the initial condition 
$$
\rho(x, 0)=\frac{1}{\sqrt{2\pi}}e^{-\frac{x^2}{2}}.
$$
Note that $W(x)$ is singular at $x=0$. We prepare the exact convolution of $W(x)$ and the basis polynomial in the neighborhood of the singularity.
\begin{figure}
\centering
\begin{subfigure}[b]{0.30\textwidth}
\includegraphics[width=\textwidth]{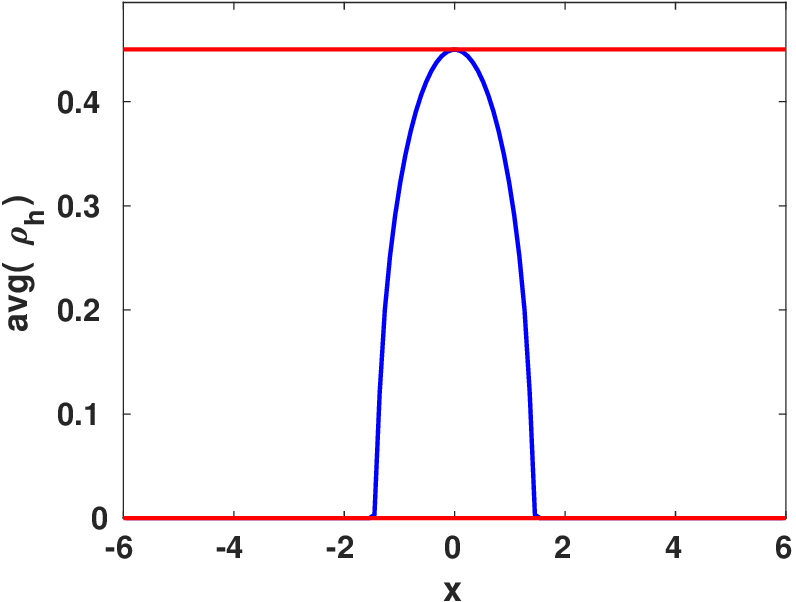}
\caption{The cell average.}
\label{Eg02cell}
\end{subfigure}
\begin{subfigure}[b]{0.30\textwidth}
\includegraphics[width=\textwidth]{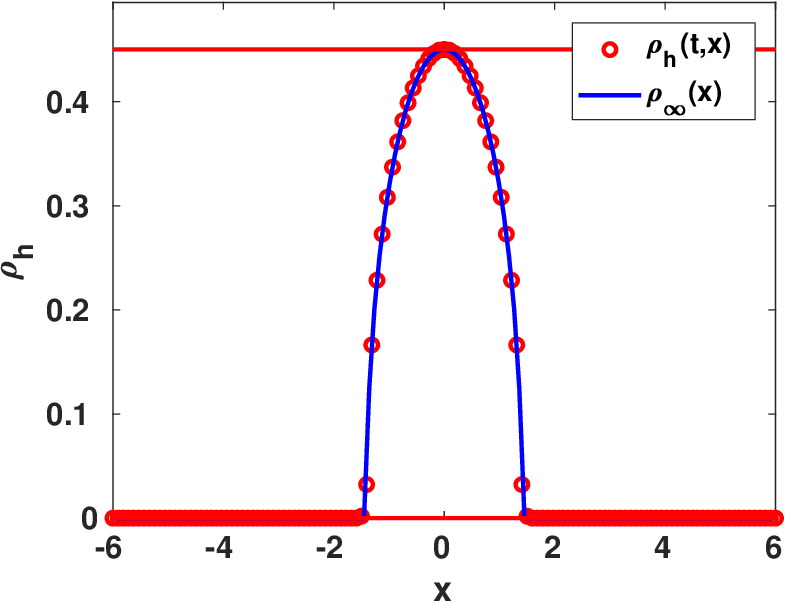}
\caption{$\rho_h(t,x)$ at $t = 10$.}
\label{Eg02uh}
\end{subfigure}
\begin{subfigure}[b]{0.30\textwidth}
\includegraphics[width=\textwidth]{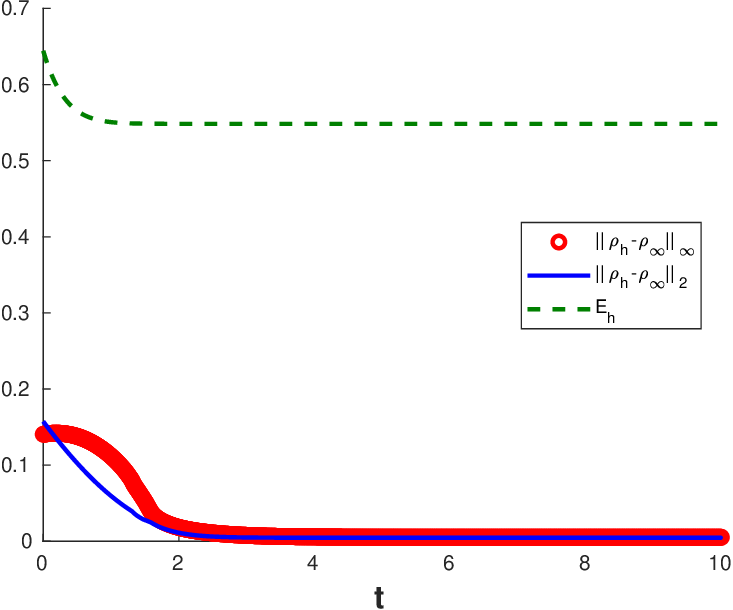}
\caption{$E_h$ and $\|\rho_h-\rho_{\infty}\|$.}
\label{Eg02energy}
\end{subfigure}
\caption{ 
$N= 128, k = 3$.
Consider the final time $t = 10$ in {\bf Example 3}.
(A) displays the cell average.
(B) illustrates  the numerical solution $\rho_h$ with a lift $\delta = 10^{-12}$.
Both preserve positivity perfectly. 
(C) shows the time evolution of the energy $E_h$ and the difference between the numerical solution $\rho_h$ and the equilibrium $\rho_{\infty}$ in $L^2$ and $L^\infty$ norms.
}
\label{figEg02}
\end{figure}
Figure \ref{figEg02}(A) illustrates the cell average of $\rho_h$ at $t=10$ with a solid blue curve.
The lower and upper red lines indicate the boundedness of the solution, which are well preserved by the cell average. Figure \ref{figEg02}(B) displays the numerical solution $\rho_h$ at $t=10$ with red circles. 
The solid blue curve represents the equilibrium $\rho_{\infty}$.
The observation that they overlap shows that the solution $\rho_h$ reaches the steady state at $t = 10$. 
Note that we lift the numerical solution with $\delta = 10^{-12}$ to enforce the positivity of the cell average needed in the positivity-preserving limiter in Sec. \ref{sec_limiter}. The lift is sufficiently small to maintain the high accuracy of the numerical solutions. 
Figure \ref{figEg02}(C) presents the discrete energy $E_h$ with dashed green lines. The energy is observed to be nonincreasing in time, satisfying the energy dissipation inequality in the discrete setting.
Meanwhile, the difference between the numerical solution $\rho_h$ and the equilibrium $\rho_{\infty}$, characterized by the $L^2$ (red circles) and $L^\infty$ (blue solid curve) norms, decreases to zero,  illustrating the convergence to $\rho_\infty$.\\

\noindent{\bf Example 4  (Nonlinear diffusion with compactly supported attraction kernel)}\\
In this example, we examine equation (\ref{IVP}) in one dimension with  $H(\rho)=\frac{\nu}{m}\rho^m$ where $m>1$,  $W(x)=W(|x|)$,  and $V(x)=0$, expressed as: 
\begin{equation}\label{mw}
\partial_t \rho= \partial_x[\rho\partial_x(\nu \rho^{m-1} + W*\rho)].
\end{equation}
This equation arises in various physical and biological models involving competing nonlinear diffusion and nonlocal attraction. Interested readers can find further studies on the equilibria of this model in references \cite{FR2011, BCLR2013, CCH14, SCS18}. 
We consider (\ref{mw}) with $\nu = 0.25, m = 3$, and the compactly supported interaction kernel: 
$$
W(x)= - (1 - |x|)_+.
$$
The equation is subject to zero flux boundary conditions and initial data given by  
\begin{equation}\label{ini_a_2}
\rho_0(x)=\frac{1}{2a}1_{[-a, a]}(x) \qquad \text{ with } a = 2 \text{ and } \Omega = [-6, 6],
\end{equation}
or
\begin{equation}\label{ini_a_3}
\rho_0(x)=\frac{1}{2a}1_{[-a, a]}(x) \qquad \text{ with } a = 3 \text{ and } \Omega = [-6, 6].
\end{equation}
These initial conditions capture the density distribution within the domain.

\begin{figure}
\centering
\begin{subfigure}[b]{0.48\textwidth}
\includegraphics[width=\textwidth]{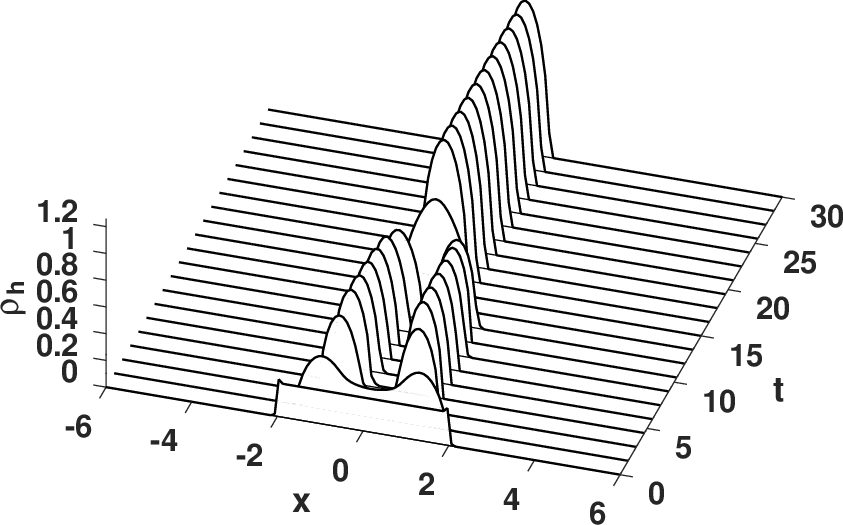}
\caption{The time evolution of $\rho_h$ for \eqref{ini_a_2}.}
\label{Eg06uh_snapshot}
\end{subfigure}\;
\begin{subfigure}[b]{0.48\textwidth}
\includegraphics[width=\textwidth]{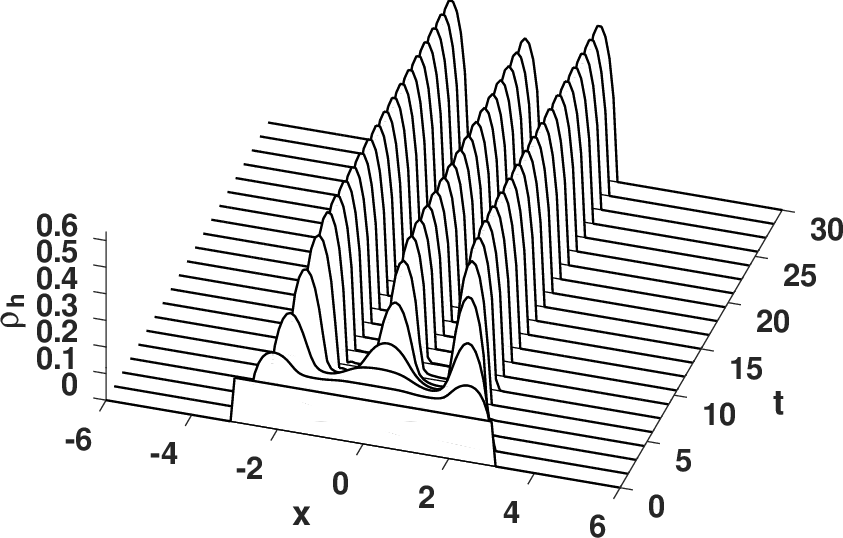}
\caption{The time evolution of $\rho_h$ for \eqref{ini_a_3}.}
\label{Eg07uh_snapshot}
\end{subfigure}\\
\begin{subfigure}[b]{0.48\textwidth}
\includegraphics[width=\textwidth]{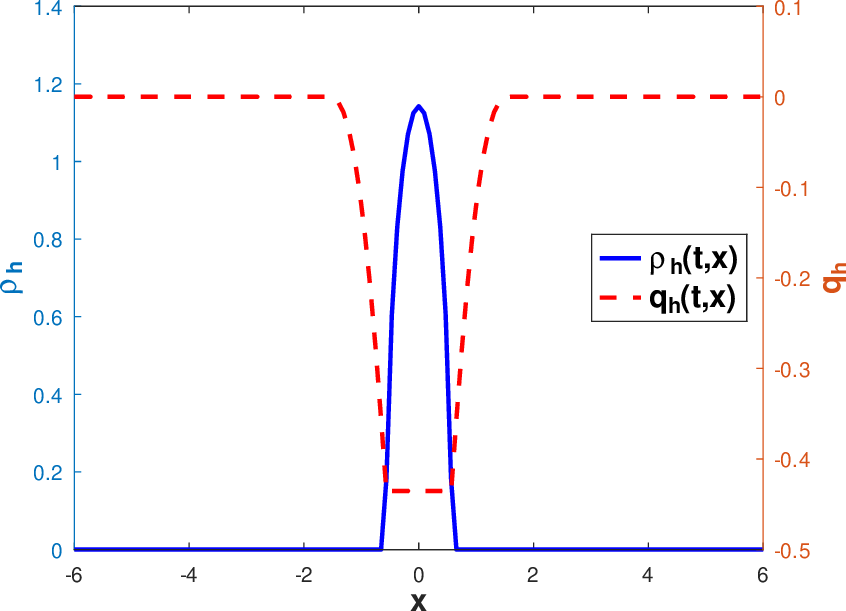}
\caption{$t = 30$ for \eqref{ini_a_2}.}
\label{Eg06uh}
\end{subfigure}\;
\begin{subfigure}[b]{0.48\textwidth}
\includegraphics[width=\textwidth]{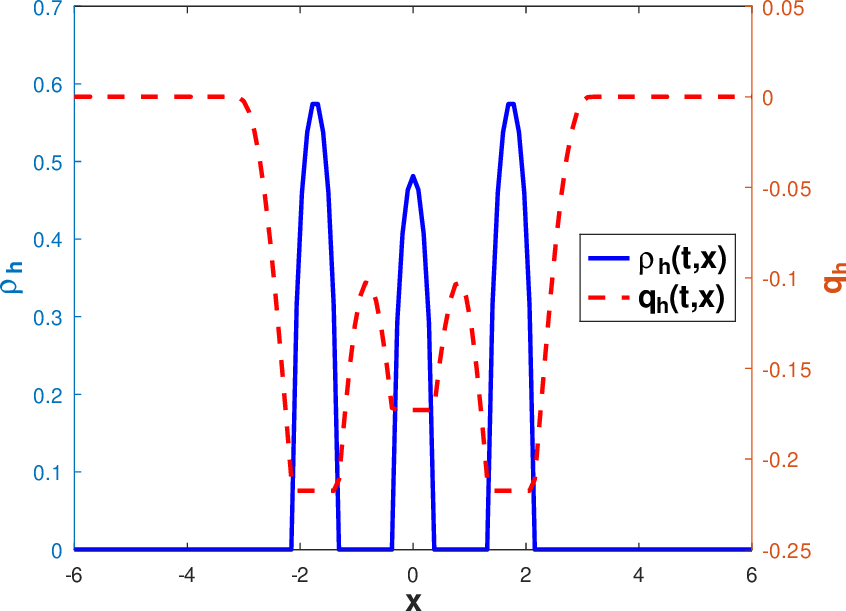}
\caption{$t = 30$ for \eqref{ini_a_3}.}
\label{Eg07uh}
\end{subfigure}
\caption{ 
$N = 128, k = 2$.
Consider the final time $t = 30$ in {\bf Example 4}.}
\label{figEg0607}
\end{figure}
Figure \ref{figEg0607} depicts the results obtained using two initial data sets \eqref{ini_a_2} and \eqref{ini_a_3}.
Figure \ref{figEg0607} (A, C) illustrate the time evolution of $\rho_h$ for $t\in[0,30]$ with solid black curves. Figure \ref{figEg0607} (B, D) present the numerical solution $\rho_h$ and the flux $q_h(t,x)$ at $t=30$ with blue solid curves and red dashed curves, respectively.  Notably, Figure \ref{figEg0607} (D) indicates that $q_h(t,x)$ attains different constants at each bump or connected component of the support of the density function $\rho_h(t,x)$, consistent with the theoretical steady states of equation \eqref{mw}.

Observing the initial data \eqref{ini_a_2}, the numerical solution $\rho_h(t,x)$ in Figure \ref{figEg0607} (A) demonstrates minimal change during the time interval $[5,10]$. 
This behavior aligns with the time evolution of the numerical entropy shown in Fig. \ref{figEhEg0607}.
\begin{figure}
\begin{minipage}[c]{0.6\textwidth}
   \includegraphics[width=\textwidth]{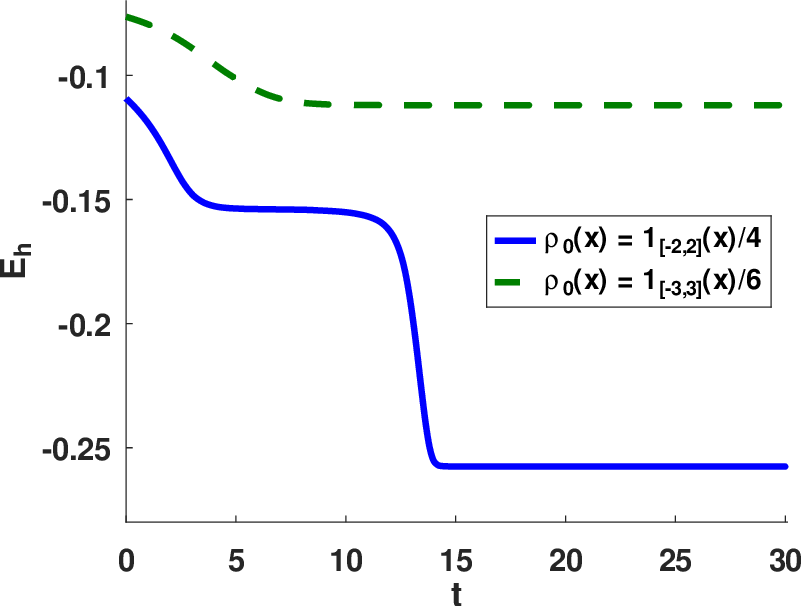}
\end{minipage}\hfill
\begin{minipage}[c]{0.4\textwidth}
\caption{
      $N = 128, k = 2$.
The time evolution of the energy $E_h$ in {\bf Example 4} for two initial data \eqref{ini_a_2} and \eqref{ini_a_3} in green dashed curve and blue solid curve, respectively.
} 
\label{figEhEg0607}  
\end{minipage}
\end{figure}
The blue solid curve in Fig. \ref{figEhEg0607} corresponds to the initial data \eqref{ini_a_2}.
Initially, it drops to a certain value and decreases slowly during the time interval $[5,10]$ until the merging of the bumps happens. Subsequently, it undergoes a significant decrease and reaches its minimum around $t =15$.\\

\noindent{\bf Example 5 (Nonlinear diffusion with Gaussian-type attraction kernel)}\\
Consider the model problem \eqref{IVP} with $V\equiv 0$ and $W(x)= - e^{-\frac{|x|^2}{4}}$.
We begin with the initial data:  
\begin{equation}
\rho_0(x) = \frac{1}{5}1_{[-5, -4]\cup [-2,1]\cup[3,4]}(x) \qquad \text{ for } x \in\Omega = [-8. 8].
\end{equation}
Figure \ref{Eg0813T800}(A) illustrates the time evolution of the numerical solution.  The density function $\rho_h(t,x)$ forms two bumps, which gradually interact with each other. Around $t = 400$, they merge again and concentrate into a single bump on the right- hand side of the domain.
Figure \ref{Eg0813T800}(B) shows the decay of the corresponding energy $E_h$, which eventually stabilizes. 
Figure \ref{Eg0813T800}(C-D) display the numerical solution $\rho_h(t,x)$ in blue solid curves and the flux $q_h(t,x)$  in red dashed curves at $t=0, 600$. 
Its evident that the density function ultimately  concentrates into a single bump, while the values of the flux $q_h(t,x)$ (in red dashed curve) remain constant. Particularly, Figure \ref{Eg0813T800}(D) provides a zoomed-in view of the numerical flux $q_h(t,x)$ within the circled spacial domain $[0.2,1.4]$, clearly demonstrating its constancy over time.\\

\begin{figure}
\centering
\begin{subfigure}[b]{0.50\textwidth}
\includegraphics[width=\textwidth]{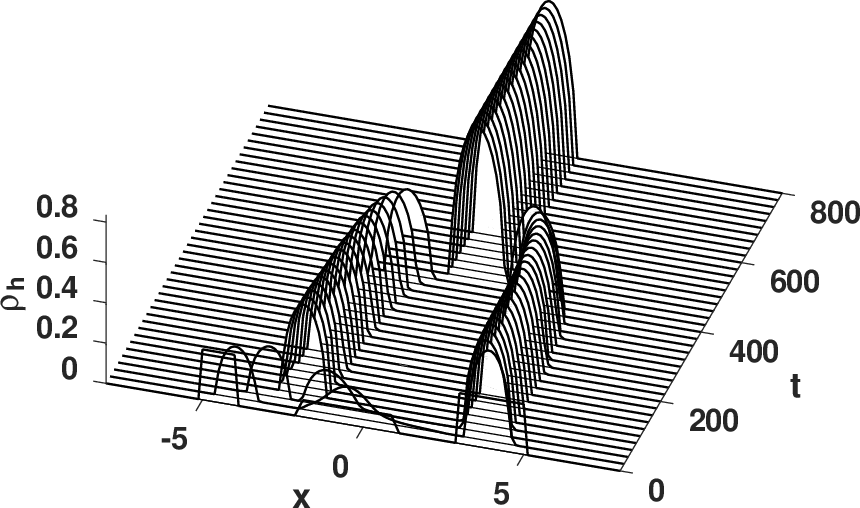}
\caption{$t \in[0,600]$.}
\end{subfigure}
\begin{subfigure}[b]{0.40\textwidth}
\includegraphics[width=\textwidth]{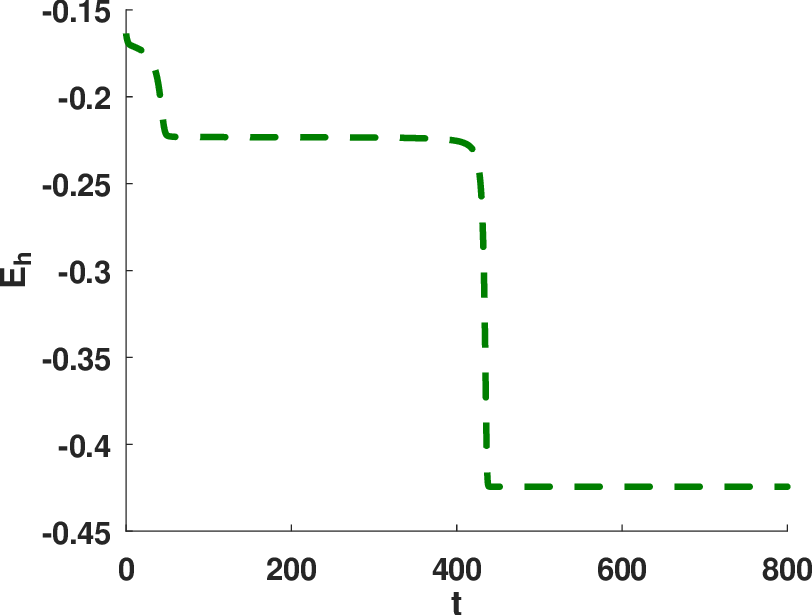}
\caption{$t \in [0, 600]$.}
\end{subfigure}\\
\begin{subfigure}[b]{0.45\textwidth}
\includegraphics[width=\textwidth]{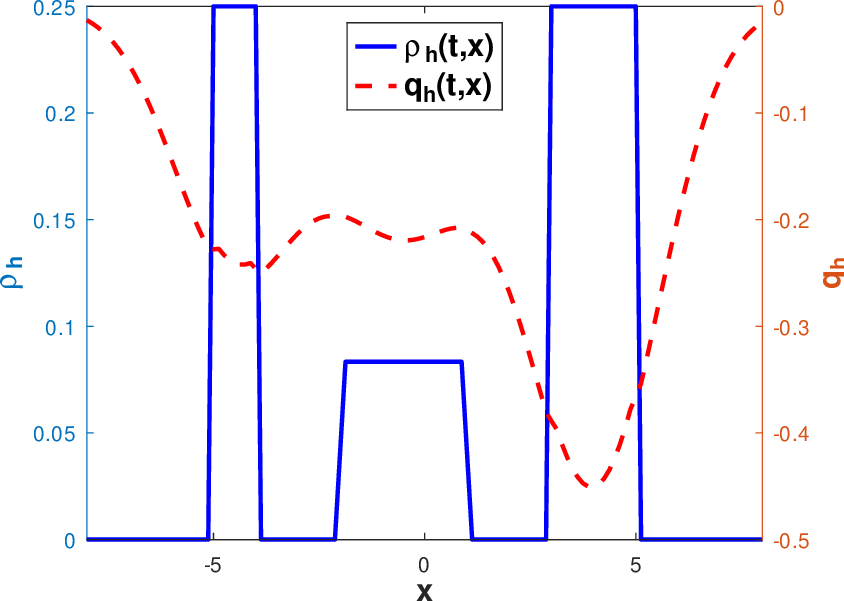}
\caption{$t = 0$.}
\end{subfigure}
\begin{subfigure}[b]{0.45\textwidth}
\includegraphics[width=\textwidth]{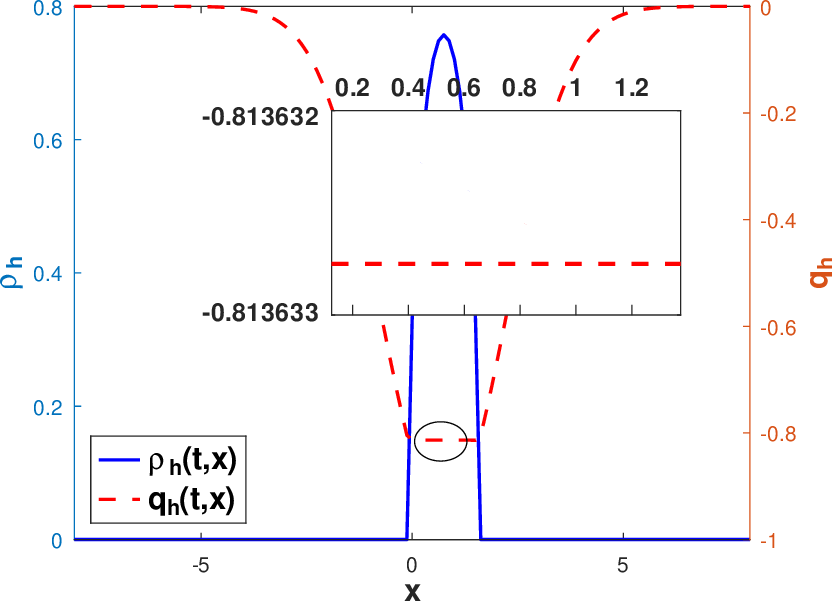}
\caption{$t = 600$.}
\end{subfigure}
\caption{ Consider 
$N = 128$ and $k = 2$ for the final time $t = 600$ in {\bf Example 5}.  The numerical solution $\rho_h$ and the flux $q_h(t,x)$ are depicted by the blue solid curve and red dashed curve, respectively.
Figure (D) provides a zoomed-in view of the numerical flux $q_h(t,x)$ within the circled spacial domain $[0.2,1.4]$. It is noticeable that the flux reaches a constant value.
}
\label{Eg0813T800}
\end{figure}

It seems that the energy $E_h(t)$ undergoes minimal changes at $t=100, 250, 500$, as shown in Figure \ref{Eg0813T800}(B).
However, a closer examination of the numerical flux
suggests that the numerical solution is significantly distant from equilibrium.
Figure \ref{Eg0813T800comparison} displays the numerical solutions $\rho_h(t,x)$ and the numerical flux $q_h(t,x)$ at $t = 100, 250, 500, 800$. Within each figure, there is a zoomed-in view of the numerical flux $q_h(t,x)$ within the circled spatial domain surrounding mass concentration.
At $t=100, 250, 500$ in Figure \ref{Eg0813T800comparison}(A-C), the flux exhibits  non-constant behavior. Figure \ref{Eg0813T800comparison}(D), along with Figure \ref{Eg0813T800}(D), illustrates that the system reaches equilibrium before $t=600$ and maintains stability over the long time.

\begin{figure}
\centering
\begin{subfigure}[b]{0.45\textwidth}
\includegraphics[width=\textwidth]{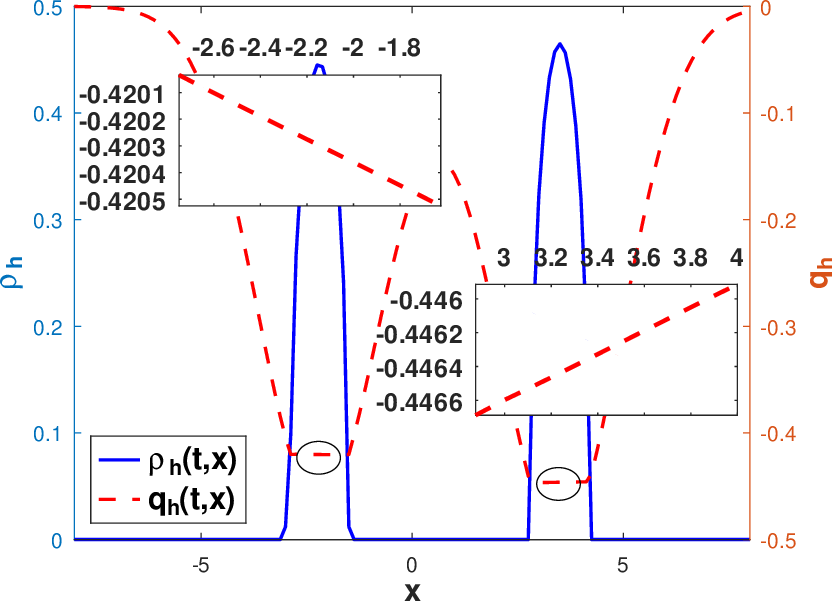}
\caption{$t = 100$.}
\end{subfigure}
\begin{subfigure}[b]{0.45\textwidth}
\includegraphics[width=\textwidth]{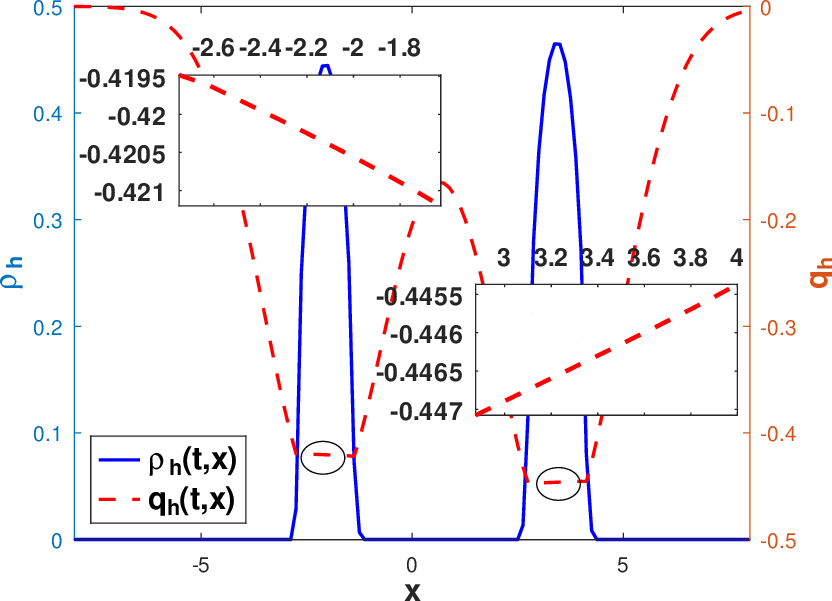}
\caption{$t = 250$.}
\end{subfigure}\\
\begin{subfigure}[b]{0.45\textwidth}
\includegraphics[width=\textwidth]{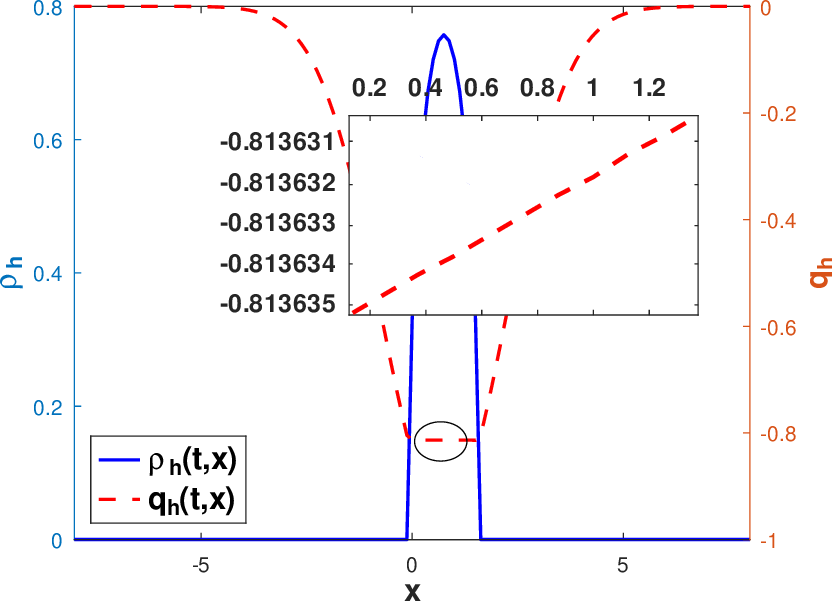}
\caption{$t = 500$.}
\end{subfigure}
\begin{subfigure}[b]{0.45\textwidth}
\includegraphics[width=\textwidth]{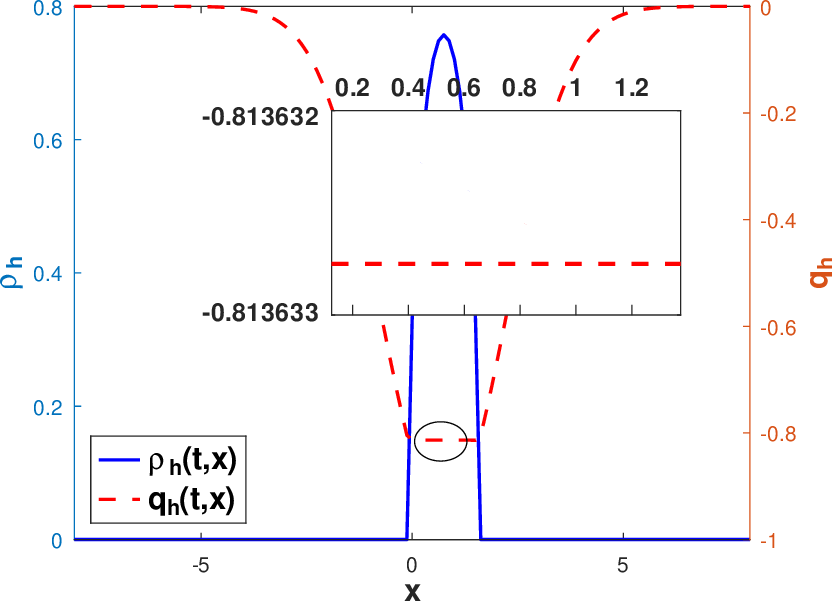}
\caption{$t = 800$.}
\end{subfigure}
\caption{ 
For $N = 128, k = 2$ in {\bf Example 5}, we consider 
the final time $t = 800$.  The numerical solution $\rho_h$ and the flux $q_h(t,x)$ are plotted in the blue solid curve and red dashed curve, respectively.
The zoomed-in view of the numerical flux $q_h(t,x)$ within the circled spatial domain surrounding mass concentration reveals its non-constant behavior at $t=100, 250, 500$. However, after reaching  equilibrium, the flux becomes constant in the neighborhood of the density bump,  as shown in Figure (D). 
}
\label{Eg0813T800comparison}
\end{figure}

\mbox{}\\

\noindent{\bf Example 6  (Nonlinear diffusion with asymmetric double well potential in 1D)}\\
Consider the porous medium nonlinear diffusion with $H(\rho)=\rho^2$ and $V(x)=x^4+0.4 x^3-5x^2$. 
Note that $V(x)$ represents an asymmetric double well potential, illustrated in Fig. \ref{figEg0910} (A).
Two different initial data sets 
\begin{equation}\label{Gaussian_0}
\rho_0(x) = \frac{1}{\sqrt{2\pi}}e^{-\frac{x^2}{2}}, \qquad x\in \Omega = [-4,4],   
\end{equation}
and
\begin{equation}\label{Gaussion_1_5}
\rho_0(x) = \frac{1}{\sqrt{2\pi}}e^{-\frac{(x-1.5)^2}{2}}, \qquad x\in \Omega = [-4,4] 
\end{equation}
are simulated for $t \in [0,10]$. 
Both initial distributions have a mass around $1$, yet their time evolution varies significantly.
The symmetric initial data \eqref{Gaussian_0} results in a mass concentration consistent with the two local minima of the double well potential $V(x)$.
Conversely, the asymmetric initial data \eqref{Gaussion_1_5} reaches an equilibrium state with a much greater mass on the same side as the initial data.\\

\begin{figure}
\centering
\begin{subfigure}[b]{0.30\textwidth}
\includegraphics[width=\textwidth]{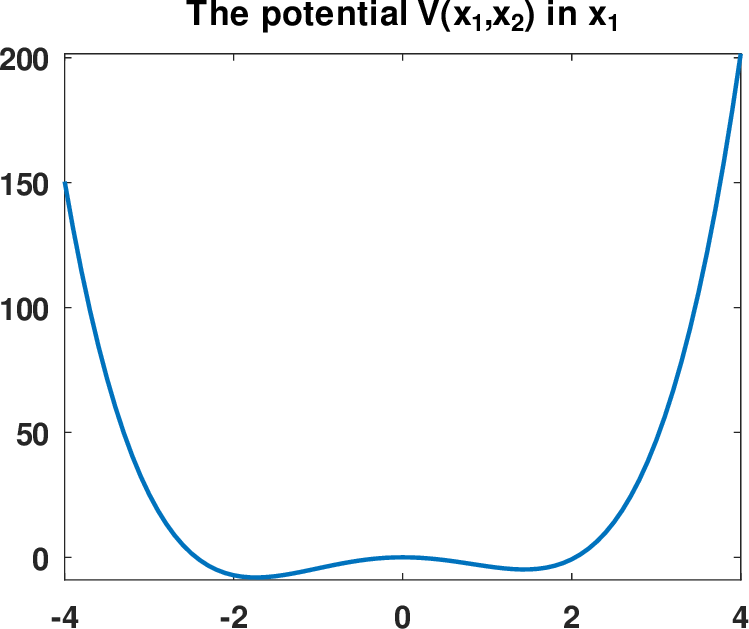}
\caption{$V(x)$.}
\label{Eg0910Vpotential1D}
\end{subfigure}
\begin{subfigure}[b]{0.30\textwidth}
\includegraphics[width=\textwidth]{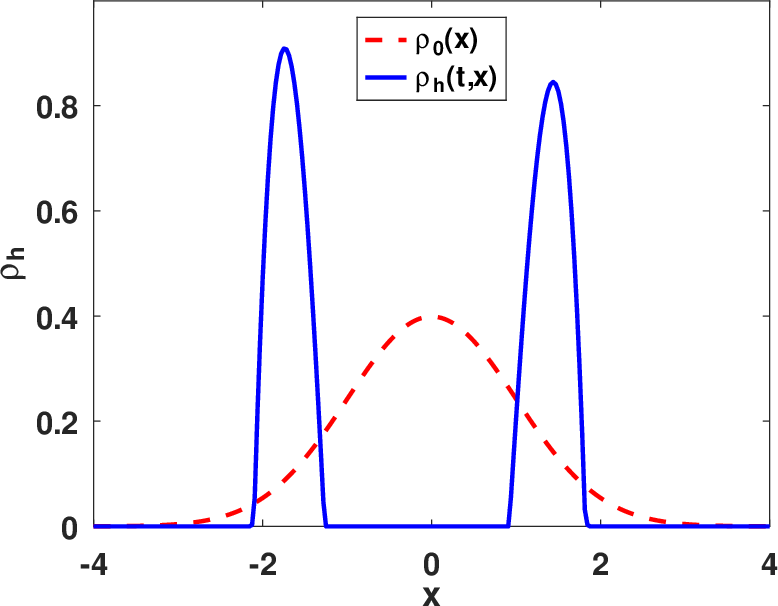}
\caption{$\rho_h(t,x)$ at $t = 0, 10$.}
\label{Eg09uhqh}
\end{subfigure}
\begin{subfigure}[b]{0.30\textwidth}
\includegraphics[width=\textwidth]{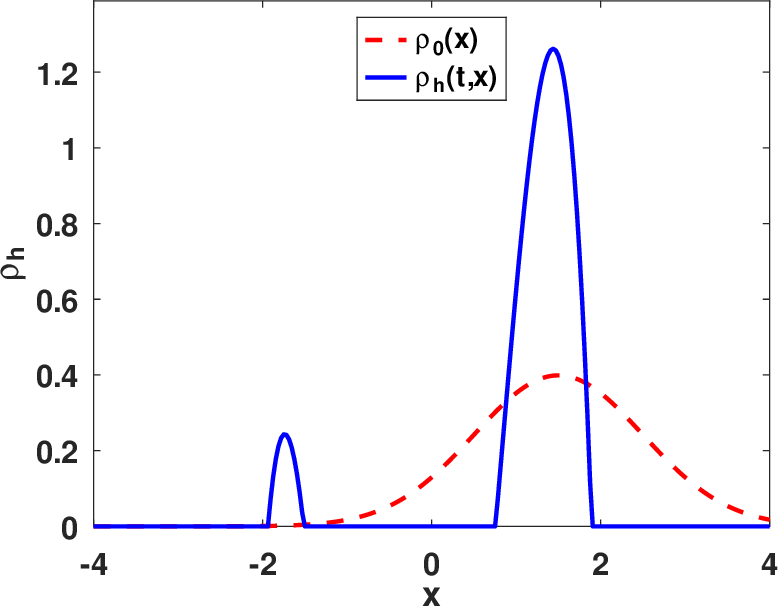}
\caption{$\rho_h(t,x)$ at $t = 0, 10$.}
\label{Eg10uhqh}
\end{subfigure}
\caption{ 
For {\bf Example 6} with $N = 256, k = 2$, consider 
the final time $t = 10$.
(A) displays the asymmetric double-well potential $V(x)$. 
(B) illustrates the numerical solution $\rho_h$ with the symmetric initial data \eqref{Gaussian_0}.
(C) showcases the numerical solution $\rho_h$ with the asymmetric initial data \eqref{Gaussion_1_5}.
In both (B) and (C), the initial data $\rho_0(x)$ are plotted with red dashed curves, and the numerical solution $\rho_h(t=60,x)$ are with blue solid curves.
}
\label{figEg0910}
\end{figure}
Fig. \ref{Eg0910uhqh}(A-B) show the numerical density function $\rho_h(t,x)$ in blue solid curves and the flux $q_h(t,x)$ in red dashed curves at $t=0, 10$ for the initial data \eqref{Gaussian_0}.
Fig. \ref{Eg0910uhqh}(C-D) illustrate the same for the initial data \eqref{Gaussion_1_5}.
Due to the large magnitude of the double-well potential $V(x)$ and the small support of $\rho_h(t,x)$, we zoom in and plot the numerical density function $\rho_h(t,x)$ and flux $q_h(t,x)$ in the domain $[-2.5, 2.5]$. It's observable that the flux $q_h(x)$ remains  constant at the aggregation location.
\begin{figure}
\centering
\begin{subfigure}[b]{0.45\textwidth}
\includegraphics[width=\textwidth]{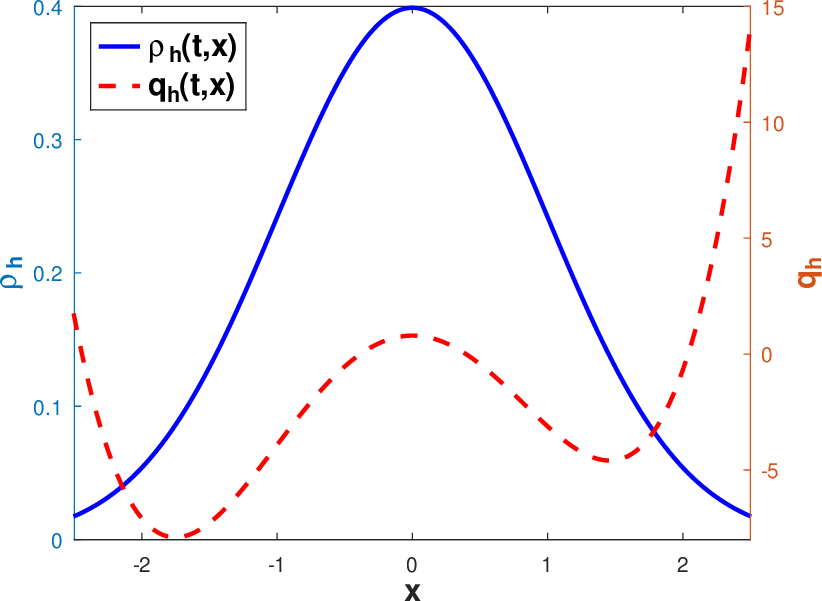}
\caption{$t = 0$.}
\end{subfigure}
\begin{subfigure}[b]{0.45\textwidth}
\includegraphics[width=\textwidth]{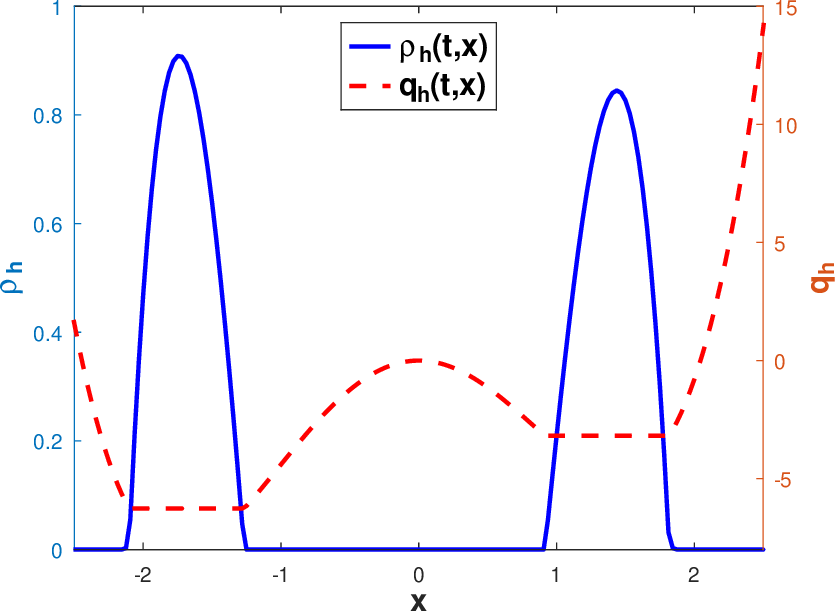}
\caption{$t = 10$.}
\end{subfigure}\\
\begin{subfigure}[b]{0.45\textwidth}
\includegraphics[width=\textwidth]{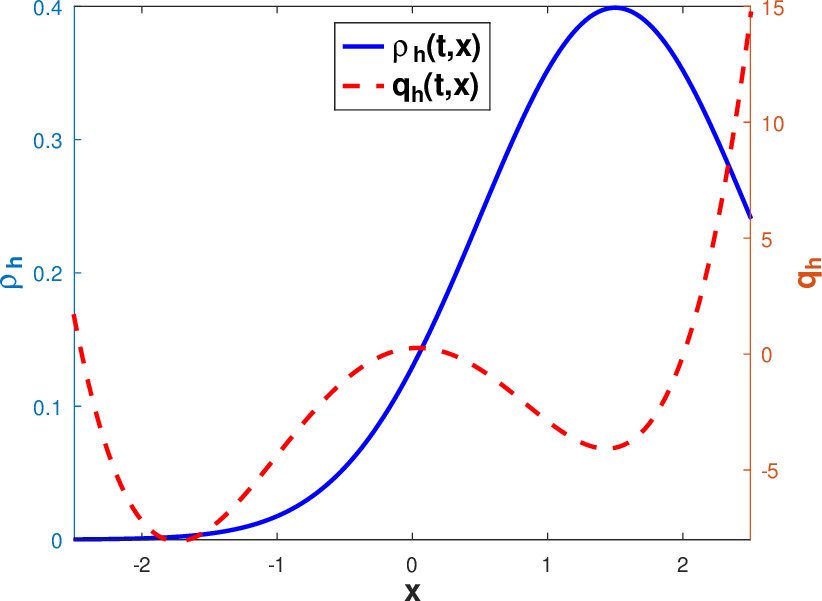}
\caption{$t = 0$.}
\end{subfigure}
\begin{subfigure}[b]{0.45\textwidth}
\includegraphics[width=\textwidth]{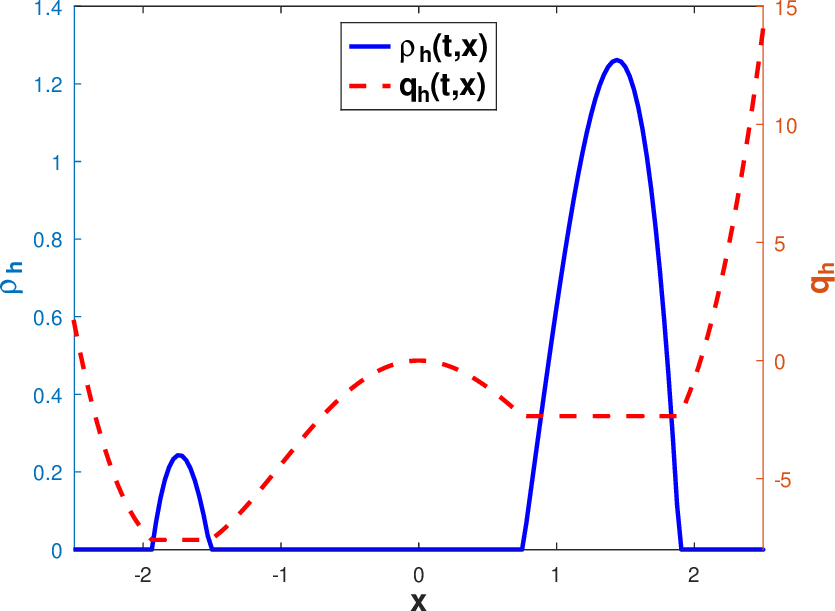}
\caption{$t = 10$.}
\end{subfigure}
\caption{ 
For {\bf Example 6} with $N = 256, k = 2$, consider 
the final time $t = 10$. The numerical solution $\rho_h$ and the flux $q_h(t,x)$ are plotted in blue solid curves and red dashed curves, respectively.
(A-B) correspond to the initial data \eqref{Gaussian_0}.
(C-D) correspond to the initial data \eqref{Gaussion_1_5}.
}
\label{Eg0910uhqh}
\end{figure}

\section{Concluding remarks}
In this paper, we present an arbitrarily  high order DG method  designed to solve a class of nonlinear Fokker-Planck equations exhibiting  a gradient flow structure. These equations adhere to a free energy dissipation law and are 
characterized by non-negative solutions. By applying the Direct Discontinuous Galerkin (DDG) method to the reformulated system with energy flux evaluated through projection, our DG scheme satisfies a discrete energy dissipation law. Additionally, through the incorporation of a local flux correction, we demonstrate the propagation of positivity of cell averages over time. Consequently, we introduce a hybrid algorithm aimed at preserving the non-negativity of the numerical density while conserving mass and maintaining numerical steady states. We provide numerical examples to illustrate the excellent performance of the proposed DG schemes.

It is worth noting that our DG schemes are readily extendable to high dimensions. However, our preliminary numerical tests suggest that even in the two-dimensional case, computing the nonlocal interaction can incur significant computational costs. This challenge partly arises from the time step restriction, as implicitly indicated in Theorem 3.1 and 3.2, which imposes a limitation of size  $ \mathcal{O}(h^2)$. This limitaiton is evidently a drawback of explicit time discretization.  In future endeavors,  we intend to extend our DG spatial discretization to multi-dimensional problems and investigate more efficient time discretization methods to preserve the positivity of numerical solutions. This may involve coupling with other limiting techniques such as the KKT (Karush-Kuhn-Tucker) limiting approach alongside implicit time discretization \cite{VXX19}.   

\section*{Acknowledgments} JAC was supported by the Advanced Grant Nonlocal-CPD (Nonlocal PDEs for Complex Particle Dynamics: Phase Transitions, Patterns and Synchronization) of the European Research Council Executive Agency (ERC) under the European Union’s Horizon 2020 research and innovation programme (grant agreement No. 883363). JAC was also partially supported by the EPSRC grants numbers EP/V051121/1 and EP/T022132/1. HL was partially supported by the National Science Foundation under Grant DMS1812666. HY was supported under Grants NSFC 12271288 and NSFC 11971258.

\bigskip

\bibliographystyle{abbrv} 
\bibliography{references}

 \end{document}